\documentclass[a4paper,10pt]{amsart}

\usepackage[english]{babel}
\usepackage[T1]{fontenc}
\usepackage[utf8]{inputenc}
\usepackage{lmodern}
\usepackage{amsmath,amssymb}
\usepackage{enumitem}
\usepackage{xcolor}
\usepackage{tikz}
\definecolor{DarkRed}{RGB}{173,0,0}
\definecolor{LightRed}{RGB}{201,0,0}
\usepackage[
  colorlinks=true,
  linkcolor=DarkRed,
  urlcolor=LightRed,
  citecolor=LightRed,
  pdfpagemode=UseNone,
  pdfstartview=FitH,
  pdftitle={Unirational Fano threefolds from the CLOP construction},
  pdfauthor={Alex Massarenti},
  pdfkeywords={Fano threefolds, unirationality, del Pezzo fibrations, relative Jacobians, sextic double solids, quartic threefolds, double quadrics}
]{hyperref}

\numberwithin{equation}{section}
\theoremstyle{plain}
\newtheorem{thm}[equation]{Theorem}
\newtheorem{Lemma}[equation]{Lemma}
\newtheorem{Proposition}[equation]{Proposition}
\newtheorem{Construction}[equation]{Construction}

\theoremstyle{definition}
\newtheorem{Definition}[equation]{Definition}
\newtheorem{Remark}[equation]{Remark}

\newtheoremstyle{introthmstyle}
  {5pt}{5pt}{\itshape}{}{\bfseries}{.}{ }{}
\theoremstyle{introthmstyle}
\newtheorem{introthm}{Theorem}

\setlist[enumerate]{label=(\roman*),before=\normalfont,font=\normalfont}
\renewcommand{\P}{\mathbb P}
\newcommand{\C}{\mathbb C}
\newcommand{\QQ}{\mathbb Q}
\newcommand{\ZZ}{\mathbb Z}
\newcommand{\OO}{\mathcal O}
\DeclareMathOperator{\Pic}{Pic}
\DeclareMathOperator{\Bl}{Bl}
\DeclareMathOperator{\Gal}{Gal}
\DeclareMathOperator{\MW}{MW}
\DeclareMathOperator{\Hess}{Hess}
\DeclareMathOperator{\Aut}{Aut}
\DeclareMathOperator{\Bir}{Bir}
\DeclareMathOperator{\Cl}{Cl}
\DeclareMathOperator{\Sing}{Sing}

\DeclareMathOperator{\mult}{mult}
\DeclareMathOperator{\id}{Id}
\DeclareMathOperator{\GL}{GL}
\DeclareMathOperator{\Spec}{Spec}
\DeclareMathOperator{\CLOP}{CLOP}

\begin{document}

\title{Unirational Fano threefolds}

\author[Alex Massarenti]{Alex Massarenti}
\address{\textsc{Alex Massarenti}\\ Dipartimento di Matematica e Informatica, Universit\`a di Ferrara, Via Machiavelli 30, 44121 Ferrara, Italy}
\email{msslxa@unife.it}

\date{\today}

\begin{abstract}
We use the relative Abel--Jacobi construction of Cheltsov--Loginov--Orlov--Prokhorov to study unirational Fano threefolds. Our main result is a $50$-dimensional family of smooth unirational del Pezzo threefolds of degree one. We also construct explicit one-parameter families of terminal unirational sextic double solids and quartic threefolds of Picard rank one. Finally, we reinterpret the Roth--Iskovskikh unirationality construction for smooth double quadric threefolds in terms of a correspondence induced by the relative Abel--Jacobi map.
\end{abstract}

\maketitle
\tableofcontents

\section*{Introduction}\label{sec:intro}

Let $X$ be a smooth Fano threefold over an algebraically closed field of characteristic zero. If $\rho(X)=1$, then $X$ belongs to one of the $17$ deformation families of smooth Fano threefolds of Picard rank one. It is a del Pezzo threefold if $-K_X=2H$ for an ample Cartier divisor $H$, and its degree is $d=H^3$. 

Unirationality is known for the smooth members of these families with only a few exceptions. The remaining cases are quartic threefolds in $\P^4$, sextic double solids, and del Pezzo threefolds of degree one. For quartic threefolds, the unirationality of an arbitrary smooth member is still open, although B. Segre constructed smooth unirational examples \cite{Seg60}. Prior to the present work, no smooth unirational example was known for either sextic double solids or degree-one del Pezzo threefolds. The present paper settles the latter existence problem by producing smooth unirational degree-one del Pezzo threefolds. In particular, as observed in \cite{CTZ25}, among Fano threefolds of index at least two the only remaining case was that of degree-one del Pezzo threefolds.

Every smooth del Pezzo threefold of degree one is a sextic hypersurface
$X_6\subset\P(1,1,1,2,3)$, where the three weight-one coordinates generate
$H^0(X,H)$. The purpose of this paper is to construct smooth unirational
threefolds of this type.

The starting point is the construction of Cheltsov--Loginov--Orlov--Prokhorov \cite{CLOP26}. Let $Z$ be a smooth del Pezzo surface of degree at least two and let $\mathcal P\subset|-K_Z|$ be an admissible pencil in the sense of \cite[Definition~3.1]{CLOP26}. Blowing up its base locus gives a genus-one fibration whose relative Jacobian is a rational elliptic surface. Contracting its zero section produces a smooth del Pezzo surface $S$ of degree one. If $A$ is ample on $Z$ and $n=A\cdot(-K_Z)$, the relative Abel--Jacobi map $Z\dashrightarrow S$ has degree $n^2$ \cite[Lemma~3.4]{CLOP26}.

For a smooth del Pezzo surface $S$ of degree one, the lattice $K_S^\perp$ is $E_8(-1)$. A line $E\subset S_{\bar k}$ determines the root $[E]+K_S$, and its strict transform on the blow-up of the anticanonical base point is a root section of the associated rational elliptic surface. The following result characterizes the degree-two case of the construction above.

\begin{introthm}\label{intro:root}
Let $k$ be a perfect field of characteristic different from $2$ and $3$, and let $S$ be a smooth del Pezzo surface of degree one over $k$. Then $S$ is obtained from the relative Jacobian of an admissible anticanonical pencil on a smooth del Pezzo surface of degree two over $k$ if and only if $S$ carries a quadratic root line in the sense of Definition~\ref{def:rootline}. In this case there is a dominant rational map $Z\dashrightarrow S$ of degree four from a smooth del Pezzo surface $Z$ of degree two.
\end{introthm}

Every smooth del Pezzo surface of degree two over $K=\C(t)$ is $K$-unirational. Indeed, weak approximation at places of good reduction gives a rational point outside the ramification curve and the generalized Eckardt locus, and the criterion of Salgado--Testa--V\'arilly-Alvarado then gives unirationality \cite{STVA14}. Thus Theorem~\ref{intro:root} gives a unirationality criterion for degree-one del Pezzo surfaces over $K$ and for threefolds fibered by them.

Our main application is the following family. Here $a_i$ is homogeneous of degree $i$ in $x_0,x_1,x_2$, while $h_6$ is a binary sextic in $x_0,x_1$.

\begin{introthm}\label{intro:family}
Let
\begin{equation}\label{eq:main-family}
X=\{w^2=z^3+a_2(x_0,x_1,x_2)z^2+a_4(x_0,x_1,x_2)z+h_6(x_0,x_1)\}
\subset\P(1,1,1,2,3).
\end{equation}
If $X$ is smooth, then $X$ is unirational.
\end{introthm}

Consider the pencil generated by $x_0$ and $x_1$. Set $x_0=u$, $x_1=tu$, $x_2=v$, $K=\C(t)$, $A_i(u,v)=a_i(u,tu,v)$, and $h=h_6(1,t)$. Its generic member is $S= (w^2=z^3+A_2(u,v)z^2+A_4(u,v)z+hu^6)$. Over $K(\sqrt h)$ it contains the lines $E_\pm=\{z=0,\ w=\pm\sqrt h\,u^3\}$. Their strict transforms are mutually inverse sections of the anticanonical elliptic fibration, hence they give the quadratic root line required by Theorem~\ref{intro:root}.

The coefficient spaces have dimensions $6$, $15$, and $7$, so
\eqref{eq:main-family} is a $28$-dimensional affine coefficient family
before quotienting by weighted changes of coordinates. Smoothness is open,
and Section~\ref{sec:example} gives a smooth member. Acting on this family
with $\Aut(\P(1,1,1,2,3))$ produces a $50$-dimensional locus in the space
of sextic hypersurfaces. Since the space of sextic
hypersurfaces in $\P(1,1,1,2,3)$ is $\P^{63}$, this locus has codimension
$13$.

It is useful to compare our construction with the classical unirational quartic threefolds arising from Segre's method \cite{Seg60}. Segre's example is based on the presence of a rational quartic surface with separable asymptotics, and this construction was later developed to produce a $54$-dimensional family of smooth unirational quartic threefolds \cite{Mar06}. In this sense, the $50$-dimensional family constructed here may be viewed as an index-two analogue of the Segre family: in both cases unirationality is forced by imposing on the threefold a special geometric structure which persists in a positive-dimensional locus, while the general member of the ambient deformation family is not known to be unirational.

We next give two index-one applications, both in explicit one-parameter families. After restricting the parameter to a nonempty Zariski open subset, every member of either family is a terminal factorial Fano threefold of Picard rank one and is unirational. In the sextic double solid family the resulting threefolds are birational to degree-one del Pezzo fibrations and are dominated by $\P^3$ by maps of degree nine. For the distinguished member $X_2$ we also determine the singularity basket explicitly.

\begin{introthm}\label{intro:sextic}
For $\lambda\in\mathbb A^1$, let
\begin{equation}\label{eq:sextic-family}
\begin{aligned}
X_\lambda&=\{w^2=g_{6,\lambda}(x_0,x_1,x_2,x_3)\}\subset\P(1_{x_0},1_{x_1},1_{x_2},1_{x_3},3_w),\\
g_{6,\lambda}={}&4x_0(x_0^2-x_1^2)x_3^3
+4(x_0^2-x_1^2)\big((\lambda x_0+x_1)^2x_2-3x_0x_2^2\big)x_3-4x_0x_2^5+(\lambda x_0+x_1)^2x_2^4\\
&-18x_0(x_0+2x_1)(\lambda x_0+x_1)x_2^3
+\big(4(x_0+2x_1)(\lambda x_0+x_1)^3-27x_0^2(x_0+2x_1)^2\big)x_2^2\\
&-4x_0(x_0^2-x_1^2)^2x_2.
\end{aligned}
\end{equation}
There is a nonempty Zariski open subset $U_6\subset\mathbb A^1$ containing $2$ such that, for every $\lambda\in U_6$, the sextic double solid $X_\lambda$ is a terminal factorial Fano threefold with
$\Cl(X_\lambda)=\Pic(X_\lambda)=\ZZ[\OO_{X_\lambda}(1)]$ and $\rho(X_\lambda)=1$. Moreover, $X_\lambda$ is birational to a degree-one del Pezzo fibration and admits a dominant rational map $\P^3\dashrightarrow X_\lambda$ of degree nine. The member $X_2$ has singularity basket $cD_4+2cA_3+13cA_1$, twelve of the $cA_1$ points being ordinary double points.
\end{introthm}

The quartics $X_{4,\kappa}$ of \eqref{eq:quartic-family} behave in the same way: after restricting to a nonempty Zariski open subset of the parameter line they are terminal factorial Fano threefolds of Picard rank one, unirational and birational to degree-one del Pezzo fibrations. For the distinguished member $X_{4,1}$ we determine the singularity basket explicitly.

\begin{introthm}\label{intro:quartic}
Fix $\iota\in\C$ with $\iota^2=-3$. For $\kappa\in\mathbb A^1$, let
\begin{equation}\label{eq:quartic-family}
\begin{aligned}
X_{4,\kappa}&=\{f_\kappa=0\}\subset\P^4_{[x_0:x_1:x_2:x_3:x_4]},\\
f_\kappa={}&x_0^2x_4^2-36\iota x_1^2x_2x_4-72\iota x_0x_2^2x_4+12\iota x_2^3x_4-x_0x_3^3-9x_1^2x_3^2-54x_0x_2x_3^2-9x_2^2x_3^2\\
&+324x_0x_1^2x_3-270x_1^2x_2x_3-945x_0x_2^2x_3-54x_2^3x_3+2916x_1^4-729\kappa^2x_0^3x_1\\
&+7776x_0x_1^2x_2-2997x_1^2x_2^2-5400x_0x_2^3+243x_2^4.
\end{aligned}
\end{equation}
There is a nonempty Zariski open subset $U_4\subset\mathbb A^1$ containing $1$ such that, for every $\kappa\in U_4$, the quartic $X_{4,\kappa}$ is a terminal factorial Fano threefold with
$\Cl(X_{4,\kappa})=\Pic(X_{4,\kappa})=\ZZ[\OO_{X_{4,\kappa}}(1)]$ and $\rho(X_{4,\kappa})=1$. Moreover, $X_{4,\kappa}$ is unirational and is birational to a fibration over $\P^1$ whose generic fiber is a smooth del Pezzo surface of degree one. The member $X_{4,1}$ has singularity basket $cD_4+2cA_2$.
\end{introthm}

We finally investigate double covers of quadrics. A classical construction due to L. Roth, reproduced by V. A. Iskovskikh, proves the unirationality of a smooth double cover of $Q^3$ ramified along a quartic section \cite{Roth50,Isk80}. Section~\ref{sec:double-quadric} shows that the point correspondence in that construction is induced by the same relative Abel--Jacobi mechanism, and the relevant fiberwise map has degree nine.
\begin{introthm}\label{intro:double-quadric}
Let $X\to Q^3$ be a smooth double cover ramified along a smooth divisor in $|\OO_Q(4)|$. Then $X$ is unirational. The Roth--Iskovskikh correspondence is induced by a degree-nine relative Abel--Jacobi correspondence.
\end{introthm}

\subsection*{Organization of the paper}
Sections~\ref{sec:roots}--\ref{sec:criterion} develop the inverse degree-two Cheltsov--Loginov--Orlov--Prokhorov construction, CLOP for short, and the criterion over $\C(t)$. Sections~\ref{sec:family}--\ref{sec:computations} treat the smooth degree-one threefolds and the explicit parametrization. Sections~\ref{globalSec} and \ref{quarticSec} give the sextic double solid and quartic families. Section~\ref{sec:double-quadric} gives the relative Abel--Jacobi interpretation of the Roth--Iskovskikh construction. Section~\ref{appendix:magma} describes the accompanying computer checks.
\section{Root lines on degree-one del Pezzo surfaces}\label{sec:roots}
We first recall the relation between lines on a degree-one del Pezzo surface and sections of its anticanonical elliptic fibration. This allows us to formulate the Galois descent condition that will characterize the degree-two CLOP construction.

Throughout Sections~\ref{sec:roots} and \ref{sec:inverse}, let $k$ be a perfect field of characteristic different from $2$ and $3$, and let $S$ be a smooth del Pezzo surface of degree one over $k$. The anticanonical pencil has a unique base point $O\in S(k)$. Set $J=\Bl_O(S)$, let $C$ be the exceptional curve, and let $\vartheta:J\to\P^1_k$ be the induced relatively minimal elliptic fibration with zero section $C$.

\begin{Lemma}\label{lem:irreducible-anticanonical}
Every geometric member of $|-K_S|$ is integral, and every geometric fiber of $\vartheta$ is irreducible.
\end{Lemma}

\begin{proof}
Work over the algebraic closure $\bar k$ of $k$, and let $D\in|-K_S|$. Write
$D=\sum_i m_iD_i$, where the $D_i$ are distinct irreducible curves.
Since $-K_S$ is ample, each $(-K_S)\cdot D_i$ is a positive integer,
while
$
1=(-K_S)^2=(-K_S)\cdot D
=\sum_i m_i(-K_S)\cdot D_i.
$

Hence $D$ has a single irreducible component, occurring with
multiplicity one, so $D$ is integral. It remains to check that blowing up the anticanonical base point does
not introduce an exceptional component in a fiber. The base scheme of
$|-K_S|$ has length $(-K_S)^2=1$, hence it is the reduced point $O$.
If $s_0,s_1$ locally generate the anticanonical pencil at $O$, then
$(s_0,s_1)=\mathfrak m_O$. 

Thus their classes form a basis of
$\mathfrak m_O/\mathfrak m_O^2$, and every nonzero linear combination
of $s_0$ and $s_1$ has order exactly one at $O$. Therefore every member
of $|-K_S|$ has multiplicity one at $O$, and its strict transform on
$J=\Bl_O(S)$ is integral. These strict transforms are precisely the
geometric fibers of $\vartheta$, so every geometric fiber of
$\vartheta$ is irreducible.
\end{proof}

A \emph{line} on $S_{\bar k}$ is a $(-1)$-curve, that is, a smooth rational
curve $E$ such that $E^2=-1$ and $(-K_S)\cdot E=1$. The lattice
$K_{S_{\bar k}}^\perp\subset\Pic(S_{\bar k})$ is isomorphic to $E_8(-1)$,
and $E\mapsto[E]+K_S$ is a Galois-equivariant bijection between the $240$
lines on $S_{\bar k}$ and the roots of $E_8$ \cite{Man86}.

Set
$S_{\bar k}=S\times_{\Spec(k)} \Spec(\bar k)$ and $J_{\bar k}=J\times_{\Spec(k)} \Spec(\bar k)$. We denote by
$\MW(J_{\bar k}/\P^1_{\bar k})$ the Mordell--Weil group of sections of
$\vartheta_{\bar k}:J_{\bar k}\to\P^1_{\bar k}$, with zero section
$C_{\bar k}$. Let $F$ denote the class of a fiber of
$\vartheta_{\bar k}:J_{\bar k}\to\P^1_{\bar k}$. Since all fibers of $\vartheta_{\bar k}$ are irreducible, the
Shioda map identifies the Mordell--Weil lattice with
$\langle C_{\bar k},F\rangle^\perp\subset\Pic(J_{\bar k})$, endowed with the
negative intersection pairing. For a section $P$ it is given by
$
\phi(P)=[P]-[C_{\bar k}]-(P\cdot C_{\bar k}+1)F.
$
Moreover, pull-back by $\pi_{\bar k}:J_{\bar k}\to S_{\bar k}$ identifies
$K_{S_{\bar k}}^\perp$ with $\langle C_{\bar k},F\rangle^\perp$. We use
throughout the resulting identification
$K_{S_{\bar k}}^\perp\simeq\MW(J_{\bar k}/\P^1_{\bar k})$.

\begin{Lemma}\label{lem:line-section}
Let $L/k$ be a field extension and let $E\subset S_L = S \times_{\Spec(k)} \Spec(L)$ be a line. Then $E$
avoids $O$, and its strict transform $P_E\subset J_L = J\times_{\Spec(k)} \Spec(L)$ is a section of
$\vartheta_L$ disjoint from $C_L$. Under the above identification, the root
$[E]+K_S$ corresponds to $P_E$.
\end{Lemma}

\begin{proof}
Suppose first that $E$ contains $O$. Since $E$ is smooth at $O$, its
strict transform $\widetilde E\subset J_L$ satisfies
$
F\cdot\widetilde E=(-K_S)\cdot E-\mult_O(E)=1-1=0.
$
Hence $\widetilde E$ is vertical for $\vartheta_L$. By
Lemma~\ref{lem:irreducible-anticanonical}, every geometric fiber of
$\vartheta_L$ is irreducible, so $\widetilde E$ would have to be a
fiber. This is impossible, since $\widetilde E\simeq E\simeq\P^1_L$,
whereas a fiber of $\vartheta_L$ has arithmetic genus one. Thus $E$
avoids $O$.

Consequently its strict transform $P_E$ is isomorphic to $E$, and
$
F\cdot P_E=(-K_S)\cdot E=1.
$
Therefore $\vartheta_L|_{P_E}:P_E\to\P^1_L$ has degree one and hence is
an isomorphism, so $P_E$ is a section. Since $E$ avoids the center of
the blow-up, its strict transform is disjoint from the exceptional
curve $C_L$.

After base change to an algebraic closure, the Shioda map therefore gives
$
\phi(P_E)=[P_E]-[C_{\bar k}]-F.
$
Moreover,
$
\pi_{\bar k}^*K_S=-F-C_{\bar k},
$
since $K_J=\pi^*K_S+C$ and $F=-K_J$. As $E$ avoids $O$, we also have
$\pi_{\bar k}^*[E]=[P_E]$. Hence
$
\phi(P_E)
=[P_E]-[C_{\bar k}]-F
=\pi_{\bar k}^*([E]+K_S).
$
Thus the root $[E]+K_S$ corresponds to the section $P_E$ under the
identification
$K_{S_{\bar k}}^\perp\simeq\MW(J_{\bar k}/\P^1_{\bar k})$
\cite{Shi90,CLOP26}.
\end{proof}

\begin{Definition}\label{def:rootline}
A \emph{quadratic root line} on $S$ consists of a separable extension $L/k$ of degree at most two and a line $E\subset S_L$. Let $P_E$ be the section of Lemma~\ref{lem:line-section}. No further condition is imposed if $L=k$, while, if $[L:k]=2$ and $\sigma$ is the nontrivial element of $\Gal(L/k)$, then
$
\sigma(P_E)=-P_E
$
in $\MW(J_L/\P^1_L)$.
\end{Definition}

If $[L:k]=2$, condition $\sigma(P_E)=-P_E$ says that $\sigma$ acts by $-1$ on the root $[E]+K_S$. Allowing $L=k$, a quadratic root line is therefore the same as a Galois-stable unordered root pair $\{\pm([E]+K_S)\}\simeq A_1$.

\section{The inverse CLOP construction}\label{sec:inverse}

In this section we characterize the degree-one del Pezzo surfaces arising from the degree-two CLOP construction. We first recall the construction and its relative Abel--Jacobi map, and then prove the root-line criterion.

Let $Z$ be a smooth del Pezzo surface of degree two. An anticanonical pencil $\mathcal P\subset|-K_Z|$ is \emph{admissible} if its base scheme is reduced, its generic member is smooth, and every geometric member is irreducible. Its base scheme has length two. Blowing it up gives a genus-one fibration $Y\to\P^1_k$; the relative Jacobian $J\to\P^1_k$ is a rational elliptic surface with irreducible fibers, and contracting its zero section gives a smooth del Pezzo surface of degree one \cite[Section~3]{CLOP26}.

\begin{Remark}\label{rem:relative-AJ}
We recall explicitly the relative Abel--Jacobi map in the degree-two case.
Let $Z$ be a smooth del Pezzo surface of degree two, let
$\mathcal P\subset|-K_Z|$ be an admissible pencil with geometric base locus
$\{Q_1,Q_2\}$, and let $\pi:Y\to Z$ be its blow-up. Denote by
$\theta:Y\to\P^1$ the induced genus-one fibration and by
$\vartheta:J\to\P^1$ its relative Jacobian. If $\eta$ is the generic point of
$\P^1$, the two exceptional sections determine points, still denoted by
$Q_1,Q_2$, on $Y_\eta$, and
$\pi^*(-K_Z)|_{Y_\eta}\simeq\OO_{Y_\eta}(Q_1+Q_2)$. 

Hence the relative
Abel--Jacobi map associated with $A=-K_Z$ restricts on the generic fiber to
$
Y_\eta\longrightarrow J_\eta=\Pic^0(Y_\eta),\,
P\longmapsto\OO_{Y_\eta}(2P-Q_1-Q_2).
$
If $Q_1$ and $Q_2$ are
defined only over a quadratic extension, the divisor $Q_1+Q_2$ is nevertheless
defined over the ground field, so the map descends. After choosing $Q_1$ as
the origin over a field where the two points are defined, the map becomes
$P\mapsto2P-(Q_2-Q_1)$ on $J_\eta$, that is, multiplication by $2$ followed
by a translation. In particular, it has degree $4$ \cite[Lemma~3.4]{CLOP26}.
\end{Remark}

\begin{proof}[Proof of Theorem~\ref{intro:root}]
Assume first that $S$ carries a quadratic root line $(L,E)$, and set
$P=P_E\subset J_L$. Suppose that $L=k$. By
Lemma~\ref{lem:line-section}, the sections $C$ and $P$ are disjoint.
Moreover, $C^2=-1$, while $P^2=E^2=-1$, since $P$ is the strict
transform of $E$ and $E$ avoids $O$. Thus they can be contracted
simultaneously to two distinct smooth points by a morphism
$q:J\to Z$. Since $K_J^2=0$, we have $K_Z^2=2$. If $F$ denotes the
fiber class, then $-K_J=F$, and the canonical divisor formula for $q$
gives
$
q^*(-K_Z)=F+C+P.
$

We claim that $-K_Z$ is ample. Ampleness may be checked after base
change to $\bar k$. Let $\Gamma\subset Z_{\bar k}$ be an irreducible
curve and let $D\subset J_{\bar k}$ be its strict transform. If $D$ is
horizontal for $\vartheta$, then
$q^*(-K_Z)\cdot D\geq F\cdot D>0$. If $D$ is vertical, then
Lemma~\ref{lem:irreducible-anticanonical} implies that $D$ is a fiber,
and hence
$q^*(-K_Z)\cdot D=(C+P)\cdot D=2$. Thus $(-K_Z)\cdot\Gamma>0$ for
every irreducible curve $\Gamma\subset Z_{\bar k}$. Since
$(-K_Z)^2=2$, the Nakai--Moishezon criterion shows that $-K_Z$ is
ample, so $Z$ is a del Pezzo surface of degree two.

The images of the fibers of $\vartheta$ form an anticanonical pencil
on $Z$. Indeed, if $F_\lambda$ is a fiber, then
$q^*q(F_\lambda)=F_\lambda+C+P$, which has class $q^*(-K_Z)$ by
$q^*(-K_Z)=F+C+P$. Since the strict transforms of the
members form the base-point-free pencil $|F|$ on $J$, its base scheme
on $Z$ consists precisely of the two reduced points $q(C)$ and
$q(P)$. The generic member is smooth, and every geometric member is
irreducible by Lemma~\ref{lem:irreducible-anticanonical}. Thus the
pencil is admissible. Blowing up its base scheme recovers $J$, whose
relative Jacobian is itself, and contracting the zero section $C$
recovers $S$. Hence this pencil gives the required CLOP presentation.

Suppose now that $[L:k]=2$, and let $\sigma$ be the nontrivial element
of $\Gal(L/k)$. Translation by a section of a relatively minimal
elliptic surface extends from the smooth locus to an automorphism of the
whole minimal regular model; hence translation by $P$ defines an
automorphism $t_P$ of $J_L$ over $\P^1_L$ \cite[Section~1]{Shi90}. Set
$
\tau=t_P\circ\sigma.
$
Since $\sigma(P)=-P$, we have
$\sigma\circ t_P\circ\sigma^{-1}=t_{\sigma(P)}=t_{-P}$, and therefore
$\tau^2=\id$. The divisor $2F+C+P$ is $\tau$-invariant and ample: its
intersection with $C$ and $P$ is $1$, its intersection with a fiber is $2$,
and every other irreducible curve has positive intersection with it; moreover
$(2F+C+P)^2=6$. Hence the descent datum is effective in the projective
category. Thus $\tau$ descends $J_L$ to a smooth genus-one fibration
$\theta:Y\to\P^1_k$. The induced descent datum on
$\Pic^0(J_L/\P^1_L)$ is the original one on $J$, since translation
acts trivially on $\Pic^0$. Hence the relative Jacobian of $Y$ is
$\vartheta:J\to\P^1_k$.

Moreover, $\tau(C)=P$ and $\tau(P)=C$. The simultaneous contraction
of the disjoint $(-1)$-curves $C$ and $P$ is therefore
$\tau$-equivariant and descends to a birational morphism
$Y\to Z$ over $k$. After base change to $L$, this is exactly the
contraction considered above. Consequently $Z$ is a smooth del Pezzo
surface of degree two, and the genus-one fibration on $Y$ descends to
an admissible anticanonical pencil on $Z$. Its relative Jacobian is
$J$, so the degree-two CLOP construction contracts the zero section
of $J$ and produces the original surface $S$.

Conversely, suppose that $S$ is obtained from an admissible
anticanonical pencil on a smooth del Pezzo surface $Z$ of degree two.
Let $\Sigma$ be its reduced base scheme. Since $k$ is perfect,
$\Sigma$ is split by a separable extension $L/k$ of degree at most
two. Write its two geometric points as $Q_1,Q_2$. On
$Y_L=\Bl_{\Sigma_L}(Z_L)$ the corresponding exceptional curves are
disjoint sections of the induced genus-one fibration. Taking $Q_1$
as zero identifies this fibration with its relative Jacobian $J_L$;
under this identification $Q_1$ becomes the zero section $C$ and
$Q_2$ becomes a section $P$ disjoint from $C$.

Contracting $C$ gives $S_L$. Let $E\subset S_L$ be the image of $P$.
Since $P$ is an exceptional curve of $Y_L\to Z_L$, we have
$P^2=-1$, and this self-intersection is unchanged when the disjoint
curve $C$ is contracted. Thus $E^2=-1$. Moreover,
$\pi^*(-K_S)=F+C$, where $\pi:J_L\to S_L$ contracts $C$, and hence
$
(-K_S)\cdot E=(F+C)\cdot P=1.
$
Since $P\simeq\P^1_L$, the curve $E$ is a line on $S_L$.

If $L=k$, this gives a $k$-rational root line. Suppose that
$[L:k]=2$. Then $\sigma$ exchanges $Q_1$ and $Q_2$. In the relative
Picard group, the section $P$ corresponds to the degree-zero divisor
class
$
[Q_2-Q_1]\in\Pic^0(Y_L/\P^1_L).
$
Therefore
$
\sigma(P)=[Q_1-Q_2]=-P.
$
By Lemma~\ref{lem:line-section}, $P=P_E$, so $(L,E)$ is a quadratic
root line on $S$.

Finally, set $A=-K_Z$ in \cite[Lemma~3.4]{CLOP26}. Since
$A\cdot(-K_Z)=(-K_Z)^2=2$, the relative Abel--Jacobi construction
gives a dominant rational map
$Z\dashrightarrow S$ of degree $2^2=4$.
\end{proof}

Theorem~\ref{intro:root} says that a degree-two CLOP presentation is the same as a Galois-stable unordered root pair $\{\pm\alpha\}\subset E_8$.

\section{A unirationality criterion for del Pezzo threefolds of degree one}\label{sec:criterion}

We now work over $K=\C(t)$. Combining the inverse CLOP construction with unirationality of degree-two del Pezzo surfaces over $K$, we obtain a unirationality criterion for degree-one del Pezzo threefolds.

\begin{Proposition}\label{prop:dp2-unirational}
Every smooth del Pezzo surface of degree two over $K$ is $K$-unirational.
\end{Proposition}

\begin{proof}
Let $Z/K$ be a smooth del Pezzo surface of degree two and spread it out
to a smooth projective family $\mathcal Z\to U$ over a nonempty open
subset $U\subset\P^1$. Since $Z$ is geometrically rationally connected,
the theorem of Graber--Harris--Starr gives $Z(K)\neq\varnothing$
\cite{GHS03}. Moreover, rationally connected varieties over function
fields of complex curves satisfy weak approximation at places of good
reduction \cite{HT06}.

We claim that $Z(K)$ is Zariski dense. Let $W\subsetneq Z$ be a proper
closed subset and let $\mathcal W\subset\mathcal Z$ be its closure.
After shrinking $U$, we may assume that
$\mathcal W_b\subsetneq\mathcal Z_b$ for every $b\in U$. Choose
$b\in U$ and
$q\in\mathcal Z_b(\C)\setminus\mathcal W_b$. Since
$\mathcal Z\to U$ is smooth, formal smoothness gives a local section
through $q$. Weak approximation at the place $b$ then gives a
$K$-point $p\in Z(K)$ whose reduction at $b$ is $q$. If $p$ belonged
to $W$, its specialization would belong to $\mathcal W_b$, contradicting
the choice of $q$. Hence $p\notin W$, and therefore $Z(K)$ is Zariski
dense.

Let $R\subset Z$ be the ramification divisor of the anticanonical
double cover $Z\to\P^2$, and let $\mathcal E\subset Z$ be the
generalized Eckardt locus, that is, the locus of geometric points
contained in four exceptional curves. Both are proper closed subsets
of $Z$. Hence the density proved above gives
$
p\in Z(K)\setminus(R\cup\mathcal E).
$
By the unirationality criterion of
Salgado--Testa--V\'arilly-Alvarado, a degree-two del Pezzo surface
containing such a rational point is unirational
\cite[Corollary~3.3]{STVA14}. Thus $Z$ is $K$-unirational.
\end{proof}

\begin{Lemma}\label{lem:spread}
Let $Y$ be an irreducible complex threefold with a dominant morphism $Y\to\P^1$ whose generic fiber $Y_\eta$ is a geometrically integral surface over $K$. If $Y_\eta$ is $K$-unirational, then $Y$ is unirational.
\end{Lemma}

\begin{proof}
Since $Y_\eta$ is $K$-unirational, there is a dominant rational map
$\P^2_K\dashrightarrow Y_\eta$. That is, there is an inclusion of
$K$-fields
$
K(Y_\eta)\hookrightarrow K(u,v).
$
Since $K=\C(t)$ is the function field of the base and $Y_\eta$ is the
generic fiber of $Y\to\P^1$, we have
$K(Y_\eta)=\C(Y)$. Hence
$
\C(Y)\hookrightarrow\C(t,u,v).
$
The latter is the function field of $\P^2\times\P^1$, so this inclusion
defines a dominant rational map
$\P^2\times\P^1\dashrightarrow Y$. Since
$\P^2\times\P^1$ is rational, $Y$ is unirational.
\end{proof}

Let $X$ be a smooth del Pezzo threefold of degree one, so $-K_X=2H$ and $H^3=1$. A pencil in $|H|$ has generic member a degree-one del Pezzo surface. Indeed, by adjunction, $-K_{X_\eta}=H|_{X_\eta}$ and $(-K_{X_\eta})^2=1$.

\begin{thm}\label{thm:V1-criterion}
Let $X$ be a smooth complex del Pezzo threefold of degree one and let $\Lambda\subset|H|$ be a pencil with smooth generic member $S/K$. If $S$ carries a quadratic root line, then $X$ is unirational.
\end{thm}

\begin{proof}
By Theorem~\ref{intro:root}, there are a smooth degree-two del
Pezzo surface $Z/K$ and a dominant rational map
$Z\dashrightarrow S$ of degree four. By
Proposition~\ref{prop:dp2-unirational}, the surface $Z$ is
$K$-unirational. Composing a dominant rational map
$\P^2_K\dashrightarrow Z$ with $Z\dashrightarrow S$ shows that $S$ is
$K$-unirational.

Resolve the rational map $X\dashrightarrow\P^1$ defined by the pencil
$\Lambda$. This gives an irreducible threefold $\widetilde X$
birational to $X$ together with a dominant morphism
$\widetilde X\to\P^1$ whose generic fiber is $S$. Hence
Lemma~\ref{lem:spread} shows that $\widetilde X$ is unirational.
Since unirationality is a birational invariant, $X$ is unirational.
\end{proof}

\section{The explicit \texorpdfstring{$28$}{28}-parameter family}\label{sec:family}

We apply the criterion of the previous section to an explicit $28$-parameter family of sextic hypersurfaces. The key point is that the generic fiber of a natural pencil carries a quadratic root line.

Set $\P=\P(1_{x_0},1_{x_1},1_{x_2},2_z,3_w)$ and $\mathcal A=H^0(\P^2,\OO(2))\oplus H^0(\P^2,\OO(4))\oplus H^0(\P^1,\OO(6))$. Then $\dim\mathcal A=6+15+7=28$. A point $(a_2,a_4,h_6)\in\mathcal A$ determines
\begin{equation}\label{eq:family}
X_{a_2,a_4,h_6}=\{w^2=z^3+a_2z^2+a_4z+h_6\}\subset\P,
\end{equation}
where the variables of the coefficients are as in Theorem~\ref{intro:family}.

\begin{Proposition}\label{prop:generic-fiber}
If $X=X_{a_2,a_4,h_6}$ is smooth, the pencil generated by $x_0$ and $x_1$ has smooth generic member $S = (w^2=z^3+A_2(u,v)z^2+A_4(u,v)z+hu^6)\subset\P_K(1_u,1_v,2_z,3_w)$, where $A_i(u,v)=a_i(u,tu,v)$ and $h=h_6(1,t)$. Moreover, $S$ carries a quadratic root line.
\end{Proposition}
\begin{proof}
The equation of the generic member is obtained by setting
$x_0=u$, $x_1=tu$, and $x_2=v$. We first prove that it is smooth.
Bertini's theorem gives smoothness away from the base curve
$B=X\cap\{x_0=x_1=0\}$, so it is enough to consider points of $B$.

The curve $B$ is integral. Indeed, writing
$a_2(0,0,x_2)=\alpha x_2^2$ and
$a_4(0,0,x_2)=\beta x_2^4$, its equation is
$
w^2=z^3+\alpha x_2^2z^2+\beta x_2^4z
$
in $\P(1_{x_2},2_z,3_w)$. The polynomial on the right is not a square:
a form of weight three in $x_2,z$ has the form
$\gamma x_2^3+\delta x_2z$, and its square has no $z^3$-term.
Hence the defining equation of $B$ is irreducible. In particular,
$B$ has only finitely many singular points.

Let $p\in B$ be smooth. Since $B$ is the complete intersection of
$X$, $\{x_0=0\}$, and $\{x_1=0\}$ near $p$, the differentials
$dF,dx_0,dx_1$ are linearly independent at $p$. Hence
$dF$ and $d(\lambda_0x_0+\lambda_1x_1)$ are linearly independent for
every $[\lambda_0:\lambda_1]\in\P^1$, so every member of the pencil is
smooth at $p$. If $p$ is singular on $B$, then
$dF\in\langle dx_0,dx_1\rangle$. Since $X$ is smooth, $dF\neq0$, and
there is therefore a unique
$[\lambda_0:\lambda_1]\in\P^1$ such that
$d(\lambda_0x_0+\lambda_1x_1)$ is proportional to $dF$. Thus exactly
one member of the pencil is singular at $p$. Since $B$ has only
finitely many singular points, only finitely many members can be
singular along the base curve. Therefore the generic member $S/K$ is
smooth.

We next claim that $h_6\neq0$. If $h_6=0$, then the plane
$\Pi=\{z=w=0\}\simeq\P^2$ is contained in $X$. At a point of $\Pi$
where $a_4=0$, all partial derivatives of the defining equation of
$X$ vanish: the derivative with respect to $z$ is $-a_4$, while all
the others vanish on $\Pi$. Since a homogeneous quartic on $\P^2$
has a zero, this contradicts the smoothness of $X$. Hence $h_6\neq0$,
and consequently
$
h=h_6(1,t)\in K^*.
$

Set $L=K(\sqrt h)$, with $L=K$ if $h$ is already a square. Over $L$,
the surface $S$ contains the two curves
$
E_\pm=\{z=0,\ w=\pm\sqrt h\,u^3\}.
$
The map
$
[u:v]\longmapsto[u:v:0:\pm\sqrt h\,u^3]
$
identifies $E_\pm$ with $\P^1_L$. By adjunction,
$-K_S=\OO_S(1)$, and
$\OO_S(1)|_{E_\pm}\simeq\OO_{\P^1}(1)$. Hence
$(-K_S)\cdot E_\pm=1$, and adjunction gives
$
-2=(K_S+E_\pm)\cdot E_\pm=-1+E_\pm^2.
$
Thus $E_\pm^2=-1$, so both $E_\pm$ are lines.

Consider now the anticanonical elliptic fibration obtained by blowing
up the base point of $|-K_S|$. On its smooth fibers the involution
$(z,w)\mapsto(z,-w)$ is the inversion map with respect to the zero
section. It exchanges the strict transforms of $E_+$ and $E_-$, and
therefore
$
P_{E_-}=-P_{E_+}.
$
If $[L:K]=2$, the nontrivial element $\sigma\in\Gal(L/K)$ exchanges
$E_+$ and $E_-$, so
$
\sigma(P_{E_+})=P_{E_-}=-P_{E_+}.
$
Thus $E_+$ is a quadratic root line. If $L=K$, it is a rational root
line, and the conclusion follows as well.
\end{proof}

\begin{proof}[Proof of Theorem~\ref{intro:family}]
Let $X$ be a smooth member of \eqref{eq:family}. By adjunction,
$-K_X=2H$ and $H^3=1$, so $X$ is a del Pezzo threefold of degree one.
By Proposition~\ref{prop:generic-fiber}, the pencil generated by
$x_0$ and $x_1$ has smooth generic member over $K=\C(t)$, and this
generic member carries a quadratic root line. Therefore
Theorem~\ref{thm:V1-criterion} applies and shows that $X$ is
unirational.
\end{proof}

\begin{Proposition}\label{prop:family-dimension}
Let $\mathcal U\subset\mathcal A$ be the nonempty open subset parametrizing
smooth hypersurfaces. The closure of
$\Aut(\P(1,1,1,2,3))\cdot\mathcal U$ in
$\P H^0(\P(1,1,1,2,3),\OO(6))$ has dimension $50$. Consequently it has
codimension $13$ in $\P^{63}$, and its image in the moduli space of smooth
degree-one del Pezzo threefolds has dimension $21$.
\end{Proposition}

\begin{proof}
Set $\P=\P(1,1,1,2,3)$. A graded automorphism of its homogeneous
coordinate ring has the form
$
x\mapsto Ax,\, z\mapsto az+q_2(x),\,
w\mapsto bw+\ell_1(x)z+q_3(x),
$
where $A\in\GL_3$, $q_2\in H^0(\P^2,\OO(2))$,
$\ell_1\in H^0(\P^2,\OO(1))$, and
$q_3\in H^0(\P^2,\OO(3))$. Thus the group of graded automorphisms has
dimension $9+7+14=30$, and quotienting by the one-dimensional subgroup
defining the weighted projective action gives $\dim\Aut(\P)=29$.

The subgroup preserving $\mathcal A$ has dimension $7$: the parabolic
subgroup of $\GL_3$ preserving $\langle x_0,x_1\rangle$ has dimension
$7$, there is one further compatible relative rescaling of $z$ and $w$,
and the weighted scalar subgroup removes one parameter. Therefore
$
\dim\bigl(\Aut(\P)\cdot\mathcal A\bigr)\leq 28+29-7=50.
$

For the reverse inequality, consider the integral coefficient point
$$
\begin{aligned}
a_2={}&x_0^2+2x_0x_1+3x_0x_2+5x_1^2+7x_1x_2+11x_2^2,\\
a_4={}&x_0^4+2x_0^3x_1+3x_0^3x_2+4x_0^2x_1^2
 +5x_0^2x_1x_2+6x_0^2x_2^2+7x_0x_1^3+8x_0x_1^2x_2\\
&+9x_0x_1x_2^2+10x_0x_2^3+11x_1^4+12x_1^3x_2
+13x_1^2x_2^2+14x_1x_2^3+15x_2^4,\\
h_6={}&x_0^6+2x_0^5x_1+3x_0^4x_1^2+4x_0^3x_1^3
 +5x_0^2x_1^4+6x_0x_1^5+7x_1^6,
\end{aligned}
$$
and write $F=w^2-z^3-a_2z^2-a_4z-h_6$. Let $M$ be the matrix whose
rows are indexed by the $64$ monomials of weighted degree six and whose
columns are the coefficient vectors of the following $58$ polynomials:
$
m_2z^2,\, m_4z,\, m_6,\,
x_jF_{x_i},\, zF_z,\, q_2F_z,\, wF_w,
\, x_i zF_w,\, q_3F_w.
$
Here $m_d$ runs through the monomials of degree $d$ in $x_0,x_1,x_2$,
$m_6$ runs through the binary sextic monomials in $x_0,x_1$,
$1\leq i,j\leq3$, and $q_2,q_3$ run through the quadratic and cubic
monomials. The first $28$ columns are the tangent directions to
$\mathcal A$, while the last $30$ are the infinitesimal graded
automorphism directions.

The exact computation gives
$
\operatorname{rank}M=51.
$
Since the infinitesimal weighted scalar action belongs to the last $30$
columns and acts on $F$ by a nonzero scalar multiple of $F$, one of these
$51$ affine tangent directions is the radial direction. Hence the image of
the differential in the projective space has rank $50$ at this point.
The rank of a matrix of regular functions is lower semicontinuous, so the
generic rank is at least $50$; consequently the swept locus has dimension
at least $50$. Combined with the upper bound above, its dimension is
exactly $50$.

The locus $\mathcal U$ is nonempty by Proposition~\ref{prop:explicit} and
is dense in $\mathcal A$, so its saturation has the same closure. The last
$30$ columns of $M$, corresponding to the affine graded automorphism
directions, have rank $30$. After quotienting by the radial weighted scalar
direction, the corresponding projective orbit therefore has dimension
$29$; equivalently, the stabilizer of a general member in $\Aut(\P)$ is
finite. Since $h^0(\P,\OO_{\P}(6))=64$, the codimension is $63-50=13$,
and the moduli dimension is $50-29=21$.
\end{proof}

\begin{Remark}
Every smooth sextic $X\subset\P(1,1,1,2,3)$ considered above has Picard
rank one, with $\Pic(X)=\ZZ[H]$ and $-K_X=2H$. Nevertheless, these
threefolds are not birationally rigid. Indeed, the pencil generated by
$x_0$ and $x_1$ has base curve $B=X\cap\{x_0=x_1=0\}$ of arithmetic
genus one and $H$-degree one. Resolving this pencil and running the
relative minimal model program over $\P^1$ gives a Mori fiber space whose
generic fiber is a del Pezzo surface of degree one. 

This agrees with Grinenko's
description of the Mori structures on a smooth Fano threefold of index
two and degree one: besides $X\to\operatorname{Spec}k$, they are
precisely the degree-one del Pezzo fibrations obtained by blowing up
curves of arithmetic genus one and degree one \cite{Gri04}. On the
other hand, Grinenko also proved that every birational self-map of $X$
is regular, that is, $\Bir(X)=\Aut(X)$. Thus these threefolds provide an example where the absence of
non-regular birational self-maps coexists with the failure of
birational rigidity. In particular, they are not rational.
\end{Remark}

\section{A smooth explicit example}\label{sec:example}

We now exhibit a concrete smooth member of the family and describe explicitly the degree-two del Pezzo surface giving its inverse CLOP presentation. Consider
\begin{equation}\label{eq:explicit-X}
X=(w^2=z^3-(x_0^2x_1^2+x_2^4)z+x_0^6+x_1^6)
\subset\P(1,1,1,2,3).
\end{equation}

\begin{Proposition}\label{prop:explicit}
The threefold $X$ in \eqref{eq:explicit-X} is smooth and unirational.
\end{Proposition}
\begin{proof}
Set
$
F=w^2-z^3+(x_0^2x_1^2+x_2^4)z-x_0^6-x_1^6.
$
The singular locus of $\P(1,1,1,2,3)$ consists of the $z$-point and
the $w$-point. Neither belongs to $X$, since
$F(0,0,0,1,0)=-1$ and $F(0,0,0,0,1)=1$. It remains to prove that
$X$ is quasismooth.

The partial derivatives are $F_w=2w$, $F_z=-3z^2+x_0^2x_1^2+x_2^4$, $F_{x_0}=2x_0(x_1^2z-3x_0^4)$, $F_{x_1}=2x_1(x_0^2z-3x_1^4)$, $F_{x_2}=4x_2^3z$. Suppose that they vanish at a nonzero point of the affine cone. Then
$w=0$. If $z=0$, the equations
$F_{x_0}=F_{x_1}=F_z=0$ give
$x_0=x_1=x_2=0$, a contradiction. Hence $z\neq0$. From
$F_{x_2}=0$ we obtain $x_2=0$, and then $F_z=0$ gives
$
x_0^2x_1^2=3z^2.
$
In particular, $x_0x_1\neq0$. The equations
$F_{x_0}=F_{x_1}=0$ become
$
x_1^2z=3x_0^4,\, x_0^2z=3x_1^4.
$
Dividing them gives $(x_1/x_0)^6=1$. On the other hand, the first
equation gives $z=3x_0^4/x_1^2$, and substitution into
$x_0^2x_1^2=3z^2$ yields
$(x_1/x_0)^6=27$, a contradiction. Thus the affine cone over $X$ is
smooth away from the origin, so $X$ is quasismooth. Since $X$ does
not meet the singular locus of the ambient weighted projective space,
it is smooth.

Finally, \eqref{eq:explicit-X} is a member of the family
\eqref{eq:family}, with
$a_2=0$, $a_4=-(x_0^2x_1^2+x_2^4)$, and
$h_6=x_0^6+x_1^6$. Hence $X$ is unirational by
Theorem~\ref{intro:family}.
\end{proof}

For the generic member of the pencil $[x_0:x_1]$, set $t=x_1/x_0$ and $h=1+t^6$. The generic member is
\begin{equation}\label{eq:explicit-generic}
S_t = (w^2=z^3-(t^2u^4+v^4)z+hu^6)
\subset\P_{\C(t)}(1_u,1_v,2_z,3_w).
\end{equation}
The following proposition identifies its degree-two source.

\begin{Proposition}\label{prop:auxiliary}
The surface
\begin{equation}\label{eq:auxiliary}
Z_t = (\Omega^2=hc^4+4h^3(t^2u^4+v^4-2cu^3))
\subset\P_{\C(t)}(1_u,1_v,1_c,2_\Omega)
\end{equation}
is a smooth del Pezzo surface of degree two. The pencil generated by $u$ and $v$ is admissible, and its relative Jacobian contracts to the surface \eqref{eq:explicit-generic}.
\end{Proposition}
\begin{proof}
Set
$
G=\Omega^2-hc^4-4h^3(t^2u^4+v^4-2cu^3).
$
The singular locus of $\P_K(1,1,1,2)$ consists of the point
$[0:0:0:1]$, which does not lie on $Z_t$. We first prove that $Z_t$
is quasismooth. At a singular point of its affine cone, the equations
$G_\Omega=G_v=0$ give $\Omega=v=0$, while
$G_c=G_u=0$ give
$
c^3=2h^2u^3,\, u^2(2t^2u-3c)=0.
$
If $u=0$, then $c=0$, so the point is the origin. Otherwise, after
scaling, set $u=1$. Then $c=2t^2/3$, and substitution into
$c^3=2h^2$ gives
$
4t^6=27h^2=27(1+t^6)^2,
$
which is not an identity in $K=\C(t)$. Thus the affine cone is smooth
away from the origin. Since $Z_t$ avoids the singular point of the
ambient weighted projective space, it is smooth. By adjunction,
$
-K_{Z_t}=\OO_{Z_t}(1).
$
Moreover, if $H=\OO_{Z_t}(1)$, then
$H^2=4/(1\cdot1\cdot1\cdot2)=2$. Hence $K_{Z_t}^2=2$, and since
$H$ is ample, $Z_t$ is a del Pezzo surface of degree two.

The base scheme of the pencil generated by $u$ and $v$ is obtained by
setting $u=v=0$, and is therefore given by
$\Omega^2=hc^4$. Since $h\neq0$, it consists of two reduced geometric
points. Moreover, $h=1+t^6$ has simple zeros and hence is not a square
in $K$, so these two points are conjugate over $K(\sqrt h)$.

For a finite parameter $\lambda$, the member $v=\lambda u$ is
\begin{equation}\label{eq:binary-quartic}
\omega^2=q_{t,\lambda}(c,u)
=hc^4-8h^3cu^3+4h^3(t^2+\lambda^4)u^4.
\end{equation}
It is geometrically irreducible. Indeed, if $q_{t,\lambda}$ were the
square of
$\alpha c^2+\beta cu+\gamma u^2$, then the vanishing of the
$c^3u$-coefficient, together with the nonzero $c^4$-coefficient, would
give $\beta=0$. The vanishing of the $c^2u^2$-coefficient would then
give $\gamma=0$, contradicting the nonzero $cu^3$-coefficient. The
member at infinity is
$
\omega^2=hc^4+4h^3v^4
$
and is geometrically irreducible by the same argument. Thus every
geometric member of the pencil is irreducible.

It remains to check that the generic member is smooth and identify its
Jacobian. For a binary quartic
$ac^4+bc^3u+\gamma c^2u^2+dcu^3+eu^4$, set
$
\mathcal I=12ae-3bd+\gamma^2,\,
\mathcal J=72a\gamma e+9b\gamma d-27ad^2-27b^2e-2\gamma^3.
$
For \eqref{eq:binary-quartic}, setting $M=t^2+\lambda^4$ gives
$
\mathcal I=48h^4M,\,
\mathcal J=-1728h^7,
$
and
$
\Delta=\frac{4\mathcal I^3-\mathcal J^2}{27}
=4096h^{12}(4M^3-27h^2)\neq0.
$
Hence the generic member is smooth. Together with the preceding
paragraphs, this proves that the pencil is admissible.

Its generic Jacobian has Weierstrass model
$
Y^2=X^3-27\mathcal I X-27\mathcal J.
$
After the change of variables $X=36h^2x$ and $Y=216h^3y$, this becomes
$
y^2=x^3-Mx+h
=x^3-(t^2+\lambda^4)x+h
$
\cite{Fis08}. On the other hand, the anticanonical pencil of the
degree-one del Pezzo surface \eqref{eq:explicit-generic}, restricted
to $v=\lambda u$, has exactly the same Weierstrass equation. Therefore
the relatively minimal Jacobian of the pencil on $Z_t$ is isomorphic
to the elliptic surface obtained by blowing up the anticanonical base
point of \eqref{eq:explicit-generic}. Contracting its zero section
therefore gives \eqref{eq:explicit-generic}, as claimed.
\end{proof}

\section{Explicit parametrization}\label{sec:computations}

This section is independent of the proof of Theorem~\ref{intro:family}. We first obtain an explicit unirational parametrization of the auxiliary degree-two del Pezzo surface by two successive applications of Manin's construction, and then compose it with the relative Abel--Jacobi map to obtain exact formulas for a dominant rational map to \eqref{eq:explicit-X}.

\subsection*{The auxiliary degree-two surface}

Work on the chart $u=1$ of \eqref{eq:auxiliary}. Let $t,m,r$ be independent parameters, set $h=1+t^6$, and write $f(v,c)=hc^4+4h^3(t^2+v^4-2c)$. The surface contains $p_0=(t^2,0,2h^2t)$. Along $(v,c)=(t^2+s,ms)$, the quadratic Taylor approximation to $\sqrt f$ at $p_0$ is $T_0(s)=2h^2t+2h(2t^6-m)s/t+(6ht^6-(2t^6-m)^2)s^2/t^3$. Set
$$
\begin{aligned}
A_0={}&m^4(1-t^6-t^{12})-8t^6m^3+12(t^{12}-t^6)m^2+16(t^{18}+3t^{12})m+12t^{18}+24t^{12}-4t^6,\\
B_0={}&4t^2h\bigl(m^3-6t^6m^2+6(t^{12}-t^6)m+4(t^{12}-t^6)\bigr).
\end{aligned}
$$
Direct expansion gives $T_0(s)^2-f(t^2+s,ms)=t^{-6}s^3(B_0+A_0s)$. Thus the first Manin step gives $s_0=-B_0/A_0$, $v_0=t^2+s_0$, and $c_0=ms_0$. Setting
$$
Q_0=c_0^2-4t^6v_0c_0+2t^2(3t^6+1)c_0-2t^6(t^6+3)v_0^2+8t^8v_0-2t^4(3t^6+1),\qquad
\omega_0=-\frac{Q_0}{t^3},
$$
we have $\omega_0=T_0(s_0)$ and $\omega_0^2=f(v_0,c_0)$. All formulas below are understood on the open set where their denominators are nonzero.

For the second Manin step, set
$$
\begin{array}{lllll}
f_v=16h^3v_0^3, & f_c=4hc_0^3-8h^3, & f_{vv}=48h^3v_0^2, & f_{cc}=12hc_0^2, & \\ 
\omega_v=\frac{f_v}{2\omega_0}, & \omega_c=\frac{f_c}{2\omega_0}, & \omega_{vv}=\frac{f_{vv}}{2\omega_0}-\frac{f_v^2}{4\omega_0^3}, & \omega_{vc}=-\frac{f_vf_c}{4\omega_0^3}, & \omega_{cc}=\frac{f_{cc}}{2\omega_0}-\frac{f_c^2}{4\omega_0^3}.
\end{array} 
$$
then
$$
\begin{array}{lll}
a =\omega_v+r\omega_c, & b =\frac{\omega_{vv}+2r\omega_{vc}+r^2\omega_{cc}}2, &  \\ 
F_3 = 24hc_0r^3+96h^3v_0, & F_4=24hr^4+96h^3, &  \\ 
C_3 = 2ab-\frac{F_3}{6}, & C_4=b^2-\frac{F_4}{24}, & s_1=-\frac{C_3}{C_4}, \\ 
v_1 = v_0+s_1, & c_1=c_0+rs_1, & \omega_1=\omega_0+as_1+bs_1^2. 
\end{array} 
$$
Hence
\begin{equation}\label{eq:second-identity}
\omega_1^2=hc_1^4+4h^3(t^2+v_1^4-2c_1).
\end{equation}
Indeed, for $T(s)=\omega_0+as+bs^2$, the difference between $T(s)^2$ and the right-hand side of \eqref{eq:second-identity} along $(v,c)=(v_0,c_0)+s(1,r)$ is $s^3(C_3+C_4s)$.

\subsection*{The $2$-covering map}

Set $M=t^2+v_1^4$ and $U(c,u)=hc^4-8h^3cu^3+4h^3Mu^4$. Let $H_U=\frac13\det\Hess(U)$ and $G_U=\frac1{12}\frac{\partial(U,H_U)}{\partial(c,u)}$. With the invariant normalization above, the covariant $2$-covering map is $[\xi:\eta:\zeta]=[-3\omega H_U:27G_U:\omega U]$ and its target is $\eta^2\zeta=\xi^3-432\mathcal I(U)\xi\zeta^2-1728\mathcal J(U)\zeta^3$ \cite{Fis08,Fis12}. This is obtained from the Jacobian model in Proposition~\ref{prop:auxiliary} by $\xi=4X$ and $\eta=8Y$. Since $\mathcal I(U)=48h^4M$ and $\mathcal J(U)=-1728h^7$, the scaling $\xi=144h^2x$ and $\eta=1728h^3y$ gives $y^2=x^3-Mx+h$. Evaluating at $(c,u,\omega)=(c_1,1,\omega_1)$ gives
$$
\begin{array}{lll}
x_E & = & \frac{4h^2(c_1^3+h^2-Mc_1^2)}{\omega_1^2}, \\ 
G_{\mathrm{red}} & = & -8M^2c_1h^2+2Mc_1^5+20Mc_1^2h^2-c_1^6-20c_1^3h^2+8h^4, \\ 
y_E & = & \frac{h^2G_{\mathrm{red}}}{\omega_1^3}.
\end{array} 
$$
Direct expansion yields $y_E^2=x_E^3-Mx_E+h$. Hence
\begin{equation}\label{eq:parametrization}
\phi:\mathbb A^3_{t,m,r}\dashrightarrow X,\qquad
(t,m,r)\longmapsto[1:t:v_1:x_E:y_E]
\end{equation}
is a rational map. The map \eqref{eq:parametrization} is dominant. It is enough to prove that $(t,m,r)\mapsto(t,v_1,c_1)$ is dominant, since the projection of the total auxiliary surface to the coordinates $(t,v,c)$ is generically finite and the $2$-covering map is finite and dominant. This can be checked by one exact specialization. After reduction modulo $101$, at $(t,m,r)=(1,1,0)$ all denominators occurring above are nonzero, $(A_0,\omega_0,C_4)=(87,74,70)$, $(v_1,c_1,\omega_1)=(56,26,79)$, and $\det\frac{\partial(v_1,c_1)}{\partial(m,r)}=45\neq0$. Thus the determinant is not the zero rational function over $\C$.

\section{Sextic double solids}\label{globalSec}
In this section we apply the relative Abel--Jacobi construction to an
explicit pencil of plane cubics over $\C(t)$. After globalizing its
relative Jacobian and running the corresponding two-ray game, we obtain
a one-parameter family of sextic double solids. We study in detail a
distinguished member, proving that it is terminal, factorial, and
unirational, and then extend the construction to a nonempty open subset
of the family.

For $\lambda\in\mathbb A^1$ set
$
e=t,\, d=t^2-1,\, c_\lambda=\lambda t+1,\, h=t+2
$
and consider on $\P^2_{\C(t)}$ the cubic pencil
$$
f_{0,\lambda}=tyz^2-(t^2-1)x^3,
\qquad
f_{1,\lambda}=y^3+(\lambda t+1)y^2z-t(t^2-1)xz^2-t^2(t+2)z^3.
$$
Its relative Jacobian globalizes exactly as below, with $c$ replaced by $c_\lambda$, and the final contraction gives the family $X_\lambda$ of \eqref{eq:sextic-family}.

We verify the distinguished member $\lambda=2$ in detail. At the end of the section we explain why, after shrinking the parameter line around $2$, the construction and the integral class-group computation extend to every member of the resulting open family.

We first record the degree-one del Pezzo surface over $K=\C(t)$ that will be globalized. Let $[x:y:z]$ be coordinates on $\P^2_K$ and consider the pencil generated by
$$
f_0=tyz^2-(t^2-1)x^3,\qquad
f_1=y^3+(2t+1)y^2z-t(t^2-1)xz^2-t^2(t+2)z^3.
$$

\begin{Lemma}\label{lem:uniform-integrality}
	Let $T$ be a noetherian scheme, let $B$ be proper over the ground field,
	and let $f:\mathcal X\to T\times B$ be a proper flat morphism of finite
	presentation. Suppose that, for some $t_0\in T$, every geometric fiber of
	$f$ over $\{t_0\}\times B$ is integral. Then there is an open
	neighborhood $T^\circ\subset T$ of $t_0$ such that every geometric fiber
	over $T^\circ\times B$ is integral.
\end{Lemma}

\begin{proof}
	The locus in $T\times B$ over which the geometric fiber of $f$ is integral
	is open \cite[Th\'eor\`eme~12.2.4]{EGAIV3}. Let $Z$ be its closed
	complement. Since $B$ is proper, the projection $T\times B\to T$ is
	proper, so the image of $Z$ is closed. It does not contain $t_0$; its
	complement is the required open neighborhood.
\end{proof}

\begin{Lemma}\label{baseirr}
The base scheme $\Sigma$ of this pencil is reduced, irreducible over $K$, and has degree nine.
\end{Lemma}

\begin{proof}
There is no base point on $z=0$. On $z=1$, eliminating $y$ gives
$
P(t,x)=(t^2-1)^3x^9+t(2t+1)(t^2-1)^2x^6-t^4(t^2-1)x-t^5(t+2).
$
For $t=3$, reduction modulo $17$ gives, up to a nonzero scalar, $x^9-8x^6-x-4$, which is irreducible. Hence $P$ is irreducible over $\QQ(t)$. The curve $P(t,x)=0$ has the smooth rational point $(-2,0)$, since $P_x(-2,0)=-48$; therefore it is geometrically irreducible, so $P$ is irreducible over $K$. Thus $\Sigma\simeq\operatorname{Spec}K[x]/(P)$ is reduced, irreducible, and of degree nine.
\end{proof}

\begin{Lemma}\label{admissiblePencil}
The pencil $\langle f_0,f_1\rangle\subset|-K_{\P^2_K}|$ is admissible.
\end{Lemma}

\begin{proof}
Write $g_\mu=f_0+\mu f_1$ on the affine chart $U=1$ of the parameter
line, where $\mu=V/U$. If $\mu\neq0$, a singular point of $g_\mu$ cannot
lie on $z=0$. After setting $z=1$, the first derivative equations reduce,
after eliminating $x$, to
$
E_1=\mu y(3y+4t+2)+t=0,
$
and
$
E_2=27\bigl(2y^3+(2t+1)y^2+t^2(t+2)\bigr)^2
     +4\mu t^3(t^2-1)^2=0.
$
Indeed, $g_x=0$ gives $3x^2+\mu t=0$, while $g_y=0$ is $E_1=0$;
squaring the equation $g_z=0$ and using these two relations gives
$\mu^2E_2=0$. Conversely, the same formulas reconstruct $x$ and give all
three first derivative equations. Thus, up to multiplication by an element
of $K^*$, the polynomial parametrizing the remaining singular members is
$
Q_8(t,\mu)=\frac{1}{27t^3}\operatorname{Res}_y(E_1,E_2),
$
which has degree eight in $\mu$.

An exact computation of Fisher's invariants \cite{Fis08,Fis12}, with a
compatible normalization of $Q_8$, gives
\begin{equation}\label{discriminant}
c_4(g_\mu)^3-c_6(g_\mu)^2
=-1728\mu^2t^3(t-1)^4(t+1)^4Q_8(t,\mu).
\end{equation}
At $t=2$, the polynomial $Q_8(2,\mu)$ is a nonzero scalar multiple of
$$
15552\mu^8+382496\mu^7+172992\mu^6-1099440\mu^5
 +1592544\mu^4-121056\mu^3+36435\mu^2-1200\mu+192.
$$
This polynomial is coprime to its derivative. Since its degree is still
eight, the resultant of $Q_8$ and $\partial Q_8/\partial\mu$ is nonzero,
and hence $Q_8$ is squarefree over $K$.

The member $f_0=0$ is an integral cuspidal cubic with unique singular point
$[0:1:0]$; on the chart $y=1$ its local equation begins with
$tz^2-(t^2-1)x^3$. The member $f_1=0$ is an integral cuspidal cubic with
unique singular point $[1:0:0]$; on the chart $x=1$ its quadratic term is
$-t(t^2-1)z^2$ and its cubic term contains $y^3$. The other singular
members correspond to the eight simple roots of $Q_8$. A simple point of
the discriminant of plane cubics represents an irreducible nodal cubic;
reducible cubics and cuspidal cubics lie in the singular locus of the
discriminant. Therefore these eight members are irreducible nodal cubics,
and all remaining members are smooth. Lemma~\ref{baseirr} gives the reduced
base scheme, so the pencil is admissible.
\end{proof}

\begin{Proposition}\label{prop:clop-p2}
The $\CLOP$ construction applied to this pencil gives a smooth del Pezzo surface $S/K$ of degree one with $\rho(S)=1$ and a dominant rational map $\P^2_K\dashrightarrow S$ of degree nine. A weighted sextic model of $S$ is
\begin{equation}\label{jacobianDP1}
y^2=4Cx^3+4d(-3eU^2V^2+c^2UV^3)x-4eU^5V+c^2U^4V^2-18ehcU^3V^3+(4hc^3-27e^2h^2)U^2V^4-4ed^2UV^5,
\end{equation}
where $e=t$, $d=t^2-1$, $c=2t+1$, $h=t+2$, $C=ed$, and $[U:V:x:y]$ have weights $1,1,2,3$.
\end{Proposition}

\begin{proof}
By Lemma~\ref{admissiblePencil}, \cite[Section~3]{CLOP26} applies. With $A=\OO_{\P^2}(1)$, the relative Abel--Jacobi map has degree $(A\cdot(-K_{\P^2}))^2=9$ \cite[Lemma~3.4 and Example~3.5]{CLOP26}. Since the base scheme is one Galois orbit, \cite[Corollary~3.8]{CLOP26} gives $\rho(S)=1$. Formula \eqref{jacobianDP1} follows from the classical Jacobian equation in terms of $c_4,c_6$ after the change $x_F=36Cx$, $y_F=108Cy$ \cite{Fis08}.
\end{proof}

Let $[r:s]$ be homogeneous coordinates on $\P^{1}$. Let
$\mathcal{P}$ be the simplicial toric variety with Cox ring
$\C[r,s,p,q,\allowbreak z,v]$
and irrelevant ideal $(r,s)\cap(p,q,z,v)$. In the basis $(F,H)$, $\deg(r)=\deg(s)=F$, $\deg(p)=H$, $\deg(q)=F+H$, $\deg(z)=F+2H$, and $\deg(v)=3F+3H$. The natural morphism $\mathcal{P}\rightarrow\P^{1}_{[r:s]}$ has fibers $\P(1,1,2,3)$. Set
$e=r$, $d=r^2-s^2$, $c=2r+s$, $h=r+2s$ and $C=ed=r(r^2-s^2)$.

\begin{Construction}\label{Yconstruction}
Let $\mathcal{Y}\subset\mathcal{P}$ be the hypersurface defined by
\begin{equation}\label{globalY}
v^2=4Cz^3+4d(-3eq^2p^2+c^2qp^3)z-4eq^5p+c^2q^4p^2-18ehcq^3p^3+(4hc^3-27e^2h^2)q^2p^4-4ed^2qp^5.
\end{equation}
Every term has class $6H+6F$. In particular, $\mathcal{Y}$ is a relative Cartier divisor, and since it contains no fiber of $\mathcal{P}\rightarrow\P^{1}$, the morphism $\pi:\mathcal{Y}\rightarrow\P^{1}$ is flat. On the chart $s=1$, the generic fiber of $\mathcal{Y}\rightarrow\P^{1}$ is equation \eqref{jacobianDP1}, with $t=r/s$, $U=q$ and $V=p$. Thus its generic fiber is the surface $S$ of Proposition~\ref{prop:clop-p2}.
\end{Construction}

\begin{Lemma}\label{fibersintegral}
Every fiber of $\pi:\mathcal{Y}\rightarrow\P^{1}$ is geometrically integral. In particular, $\mathcal{Y}$ is normal.
\end{Lemma}

\begin{proof}
A fiber has equation $v^2=F_6(p,q,z)$ in $\P(1,1,2,3)$. If $C=r(r^2-s^2)\neq0$, then $F_6$ has a nonzero $z^3$ term. A weighted form of degree three in $p,q,z$, where $\deg p=\deg q=1$ and $\deg z=2$, is at most linear in $z$, so its square has no $z^3$ term. Hence $F_6$ is not a square.

It remains to consider the three zeros of $C$. On the chart $s=1$, after multiplying by a nonzero scalar, the right hand side of \eqref{globalY} specializes to $-p^2q(-8p^2q+4pz-q^3)$, $pq^2(81p^3-162p^2q+9pq^2-4q^3)$ and $-pq^2(31p^3+18p^2q-pq^2-4q^3)$ for $r=0,1,-1$, respectively. The first has odd valuation along $q=0$, while the last two have odd valuation along $p=0$. Thus none is a square. Since the ground field has characteristic different from two, $v^2-F_6$ is irreducible in every case.

The simplicial toric variety $\mathcal P$ is Cohen--Macaulay, and
$\mathcal Y$ is an effective Cartier divisor, so $\mathcal Y$ is
Cohen--Macaulay and in particular satisfies $S_2$. Its generic fiber is
smooth. At the generic point of a special fiber, flatness over the smooth
base curve and generic reducedness of that integral fiber show that the
one-dimensional local ring is a discrete valuation ring. Hence
$\mathcal Y$ is regular in codimension one. Serre's criterion gives
normality.
\end{proof}

\begin{Proposition}\label{ClY}
One has $\Cl(\mathcal{Y})=\ZZ H\oplus\ZZ F$. In particular, $\mathcal{Y}$ is $\QQ$-factorial and $\rho(\mathcal{Y}/\P^{1})=1$.
\end{Proposition}

\begin{proof}
Let $S=\mathcal{Y}_\eta$. By Proposition~\ref{prop:clop-p2}, $\rho(S)=1$. Since $S$ is geometrically rational, $\Pic(S)$ is torsion-free, and $-K_S$ is primitive since $K_S^2=1$. Thus $\Pic(S)=\mathbb{Z}[-K_S]$. Moreover, $H|_S=-K_S$.

Let $D$ be a prime Weil divisor on $\mathcal{Y}$. There is an integer $m$ such that $D|_S\sim mH|_S$. After subtracting $mH$ and the divisor of a rational function on $\mathcal{Y}$, we may assume that $D$ is vertical. By Lemma \ref{fibersintegral}, every vertical prime divisor is a fiber of $\pi$, hence has class $F$. Therefore $\Cl(\mathcal{Y})$ is generated by $H,F$. They are independent, since $H$ has positive degree on curves contained in a general fiber, while $F$ has positive degree on any curve dominating the base. Hence $\Cl(\mathcal{Y})=\ZZ H\oplus\ZZ F$. Both classes are $\QQ$-Cartier as restrictions of toric divisor classes, so $\mathcal{Y}$ is $\QQ$-factorial. Moreover, $\rho(\mathcal{Y})=2$ and $\rho(\mathcal{Y}/\P^{1})=1$.
\end{proof}

By adjunction, the sum of the torus invariant divisors of $\mathcal{P}$ has class $7F+7H$. Since $\mathcal{Y}\sim6F+6H$, we have $-K_{\mathcal{Y}}=F+H$. Set $A=F+H$ and $L=F+2H$. The rays of the Cox coordinates, ordered by slope, are $F,A,L,H$.
The original quotient corresponds to the chamber $\langle F,A\rangle$. The two-ray game starts by crossing the wall generated by $A$ and then
the wall generated by $L$.

\begin{Proposition}\label{flop}
Let $\mathcal{P}^{+}$ be the toric variety with Cox ring
$\C[r,s,p,q,z,v]$, with the same grading as $\mathcal{P}$,
and irrelevant ideal
$(r,s,q,v)\cap(p,z).$
Thus $\mathcal{P}^{+}$ corresponds to the chamber
$\langle A,L\rangle$ of the secondary fan. Let
$\mathcal{Y}^{+}\subset\mathcal{P}^{+}$ be the hypersurface
\begin{equation}\label{globalYplus}
\begin{aligned}
v^2={}&4r(r^2-s^2)z^3
+4(r^2-s^2)
\big(-3rq^2p^2+(2r+s)^2qp^3\big)z-4rq^5p+(2r+s)^2q^4p^2
\\
&-18r(r+2s)(2r+s)q^3p^3+\big(4(r+2s)(2r+s)^3
-27r^2(r+2s)^2\big)q^2p^4-4r(r^2-s^2)^2qp^5.
\end{aligned}
\end{equation}
Then the wall crossing at $A$ induces a small crepant birational map
$\chi:\mathcal{Y}\dashrightarrow\mathcal{Y}^{+}.$
Its exceptional curves are
$
C_-=\{p=z=v=0\}\subset\mathcal{Y}
$
and
$
C_+=\{r=s=v=0\}\subset\mathcal{Y}^{+},
$
and both are isomorphic to $\P^{1}$.
\end{Proposition}

\begin{proof}
The two GIT quotients $\mathcal{P}$ and $\mathcal{P}^{+}$ have the
same Cox ring and grading. Their irrelevant ideals are
$(r,s)\cap(p,q,z,v) \qquad\text{and}\qquad (r,s,q,v)\cap(p,z),$
respectively. Hence the induced birational map is the identity on
Cox coordinates on the common stable locus.

The equation defining $\mathcal{Y}$ is homogeneous of class
$6F+6H=6A$, so the proper transform $\mathcal{Y}^{+}$ is defined by
the same Cox equation, namely \eqref{globalYplus}.

The locus lost when crossing the wall is $p=z=0$. On
$\mathcal{Y}$, equation \eqref{globalYplus} gives $v=0$, while
stability gives $q\neq0$ and $(r,s)\neq(0,0)$. Hence the exceptional
locus on $\mathcal{Y}$ is
$C_-=\{p=z=v=0,\ q\neq0,\ (r,s)\neq(0,0)\}\cong\P^{1}.$
On the other side the new exceptional locus is
$C_+=\{r=s=v=0,\ q\neq0,\ (p,z)\neq(0,0)\}\cong\P^{1}.$
Thus the wall crossing is small.

Moreover, the sum of the torus invariant divisors has class $7A$
and $\mathcal{Y}\sim6A$, so adjunction gives
$-K_{\mathcal{Y}}=A$. The same computation gives
$-K_{\mathcal{Y}^{+}}=A$. Since $A$ defines the common wall
contraction, both canonical divisors are numerically trivial on the
contracted curves. Hence $\chi$ is crepant.
\end{proof}

\begin{Proposition}\label{divcontract}
The toric variety $\mathcal{P}^{+}$ is the weighted blow-up
$\Phi:\mathcal{P}^{+}\longrightarrow \P(1_{x_0},1_{x_1},1_{x_2},1_{x_3},3_w)$
of the point
$P=[0:0:0:1:0]$
with weights $(2,2,1,3)$ in the local coordinates $(x_0,x_1,x_2,w)$.
The morphism is given by
\begin{equation}\label{ambientcontraction}
[x_0:x_1:x_2:x_3:w]
=
[rp^2:sp^2:qp:z:vp^3].
\end{equation}

Its restriction to $\mathcal{Y}^{+}$ is a divisorial contraction
$\varphi:\mathcal{Y}^{+}\longrightarrow X,$
where
\begin{equation}\label{sexticEq}
X=\{w^2=g_6(x_0,x_1,x_2,x_3)\}
\subset\P(1_{x_0},1_{x_1},1_{x_2},1_{x_3},3_w)
\end{equation}
and
$$
\begin{aligned}
g_6={}&4x_0(x_0^2-x_1^2)x_3^3
+4(x_0^2-x_1^2)\big((2x_0+x_1)^2x_2-3x_0x_2^2\big)x_3-4x_0x_2^5+(2x_0+x_1)^2x_2^4
\\
&-18x_0(x_0+2x_1)(2x_0+x_1)x_2^3+\big(4(x_0+2x_1)(2x_0+x_1)^3
-27x_0^2(x_0+2x_1)^2\big)x_2^2-4x_0(x_0^2-x_1^2)^2x_2.
\end{aligned}
$$
The exceptional divisor is
$E=\{p=0\}\cap\mathcal{Y}^{+} = \{p=0,\ v^2=4r(r^2-s^2)z^3\},$
and $\varphi(E)=P$.

The point $P\in X$ is of type $cD_4$, and
$\varphi:\mathcal{Y}^{+}\rightarrow X$ is the divisorial extraction
given locally by the weighted blow-up of $(x_0,x_1,x_2,w)$ with weights
$(2,2,1,3)$. Its discrepancy is one.
\end{Proposition}

\begin{proof}
In the basis $(L,H)$ the Cox degrees of
$r,s,p,q,z,v$ are
$
(1,-2),\, (1,-2),\, (0,1),\, (1,-1),\, (1,0),\, (3,-3).
$
Set
$
x_0=rp^2,\, x_1=sp^2,\, x_2=qp,\, x_3=z,\, w=vp^3.
$
These monomials have degrees
$
(1,0),\, (1,0),\, (1,0),\, (1,0),\, (3,0).
$
They generate the section ring of $L$, and therefore define the
toric morphism
$\Phi:\mathcal{P}^{+}\longrightarrow \P(1,1,1,1,3).$
Outside $p=0$ we can use the second torus action to set $p=1$,
and the above coordinates recover $r,s,q,z,v$. Thus $\Phi$ is an
isomorphism away from $p=0$.

On the divisor $p=0$, stability implies $z\neq0$, and
\eqref{ambientcontraction} gives
$
x_0=x_1=x_2=w=0.
$
Hence the whole divisor $p=0$ is contracted to
$P=[0:0:0:1:0]$.

On the affine chart $x_3=1$ the coordinates of the target satisfy
$x_0=rp^2,\qquad x_1=sp^2,\qquad x_2=qp,\qquad w=vp^3.$
Therefore the valuation defined by $p=0$ has weights
$
(2,2,1,3).
$
Thus $\Phi$ is the weighted blow-up of $P$ with these weights.

A direct substitution gives
$
g_6(rp^2,sp^2,qp,z)
=
p^6\,F(r,s,p,q,z),
$
where $F$ is the right hand side of \eqref{globalYplus}. Since
$w^2=v^2p^6$, the proper transform of $X$ is precisely
\eqref{globalYplus}. Hence
$\mathcal{Y}^{+}$ is the strict transform of $X$ and
$\varphi=\Phi|_{\mathcal{Y}^{+}}$.

The exceptional divisor is obtained by setting $p=0$ in
\eqref{globalYplus}, which gives
$
E=\{p=0,\ v^2=4r(r^2-s^2)z^3\}.
$
Its right hand side is not a square, so $E$ is prime. Hence
$\varphi$ is divisorial.

On the chart $x_3=1$, the equation of $X$ at $P$ is
$
w^2=4x_0(x_0^2-x_1^2)+g_{\geq5}(x_0,x_1,x_2),
$
where every monomial of $g_{\geq5}$ has ordinary degree at least
five. A general hyperplane section through $P$ has equation
$
w^2-4x_0(x_0-x_1)(x_0+x_1)+H_{\geq4}=0.
$
After splitting the nondegenerate $w$-direction, the residual plane
curve has an ordinary triple point. Hence $P$ is of type $cD_4$.

Finally, with respect to the weights $(2,2,1,3)$ the initial form of
the local equation is
$
w^2-4x_0(x_0^2-x_1^2),
$
of weighted degree six. Therefore, the discrepancy is
$
a(E,X)=2+2+1+3-1-6=1.
$
\end{proof}

We determine the singularities of the sextic double solid in \eqref{sexticEq}. 
\begin{Proposition}\label{singX}
The singular locus of $X$ is finite. Its singularity basket is
$cD_4+2cA_3+13cA_1,$
where twelve of the $cA_1$ points are ordinary double points. In particular, $X$ is terminal.
\end{Proposition}
\begin{proof}
Set
$F=w^2-g_6(x_0,x_1,x_2,x_3)$, so that
$X=\{F=0\}\subset\P(1,1,1,1,3)$.
The only singular point of the ambient weighted projective space is
$[0:0:0:0:1]$, and this point does not belong to $X$. Hence the
Jacobian criterion applies at every point of $X$.

Let $Q\in\Sing(X)$. Since $F_w=2w$, we have $w(Q)=0$. The equation
$F(Q)=0$ then gives $g_6(Q)=0$. Moreover,
$F_{x_0}=-(g_6)_{x_0},\, F_{x_1}=-(g_6)_{x_1},\, F_{x_2}=-(g_6)_{x_2},\, F_{x_3}=-(g_6)_{x_3}.$
Thus the image of $Q$ under the double cover
$X\rightarrow\P^3_{[x_0:x_1:x_2:x_3]}$ is a singular point of the
branch surface
$B=\{g_6=0\}\subset\P^3$. Conversely, every singular point
of $B$ lifts to a singular point of $X$ with $w=0$. Therefore
$\Sing(X)= \{[x_0:x_1:x_2:x_3:0]\mid [x_0:x_1:x_2:x_3]\in\Sing(B)\}.$

We first determine the singular points lying on $x_2=0$. A direct
differentiation gives
$$
\begin{aligned}
(g_6)_{x_0}|_{x_2=0}&=4(3x_0^2-x_1^2)x_3^3,\qquad
(g_6)_{x_1}|_{x_2=0}=-8x_0x_1x_3^3,\qquad (g_6)_{x_3}|_{x_2=0}&=12x_0(x_0^2-x_1^2)x_3^2,\\
(g_6)_{x_2}|_{x_2=0}&=-4(x_0^2-x_1^2)
\big(x_0^3-x_0x_1^2-4x_0^2x_3-4x_0x_1x_3-x_1^2x_3\big).
\end{aligned}
$$
Since $g_6$ is homogeneous of degree six, Euler's identity gives
$6g_6=x_0(g_6)_{x_0}+x_1(g_6)_{x_1}+x_2(g_6)_{x_2}+x_3(g_6)_{x_3}.$
Hence $g_6=0$ at every common zero of its first derivatives, and it is enough to solve the equations above.

Assume first that $x_3\neq0$. From $(g_6)_{x_1}=(g_6)_{x_0}=0$ we obtain
$x_0x_1=0$ and $3x_0^2=x_1^2$. Hence $x_0=x_1=0$, and we get
$P=[0:0:0:1:0].$
This is the $cD_4$ point described in Proposition \ref{divcontract}.

Now assume $x_3=0$. The last equation becomes
$-4x_0(x_0^2-x_1^2)^2=0$. Hence
$x_0=0$ or $x_0^2=x_1^2$, and the remaining singular points on $x_2=0$ are
$P_0=[0:1:0:0:0],\, P_+=[1:1:0:0:0],\, P_-=[1:-1:0:0:0].$

We determine their analytic types. At $P_0$, work on the chart $x_1=1$.
The quadratic part of the local equation $w^2-g_6=0$ is
$
w^2+4x_0x_2-8x_2^2+4x_2x_3,
$
which has rank three. By the splitting criterion
\cite[Corollary 2.6]{Pae24}, the germ is of type $cA_1$.

At $P_+$, work on the chart $x_0=1$ and set $x_1=1+a$. The quadratic part
of $w^2-g_6$ has rank two and is nondegenerate in the variables
$w,x_2$. By the analytic splitting lemma these variables can be split
off. Solving the $x_2$-critical equation to order two gives
$
x_2=\frac{4a(2a+9x_3)}{81}+O(3).
$
After substitution, the first nonzero homogeneous part of the
residual function in $a,x_3$ is
$
\frac{8}{81}a
\big(8a^3+72a^2x_3+162ax_3^2+81x_3^3\big),
$
which has degree four. Hence $P_+$ is of type $cA_3$ by
\cite[Corollary 2.6]{Pae24}.

The computation at $P_-$ is analogous. On the chart $x_0=1$, set
$x_1=-1+a$. After splitting the two nondegenerate directions, the first
nonzero homogeneous part of the residual function is
$
-\frac{8}{31}a
\big(8a^3-8a^2x_3+2ax_3^2+31x_3^3\big).
$
It has degree four, so $P_-$ is also of type $cA_3$.

It remains to determine the singular points with $x_2\neq0$. Work on
the affine chart $x_2=1$ and set
$
g=g_6(x_0,x_1,1,x_3),\,
I=(g,g_{x_0},g_{x_1},g_{x_3})\subset\QQ[x_0,x_1,x_3].
$
A Gr\"obner basis computation for degree reverse lexicographic
order with $x_0>x_1>x_3$ gives
$
\operatorname{in}(I)=
(x_0^3,x_0^2x_1,x_0x_1^2,x_1^3,x_0^2x_3,x_0x_3^2,x_1x_3^2,x_3^3).
$
The standard monomials are
$1,x_3,x_3^2,x_1,x_1x_3,x_1^2,x_1^2x_3,x_0,x_0x_3,x_0x_1,x_0x_1x_3,x_0^2,$ and therefore
$
\dim_{\QQ}\QQ[x_0,x_1,x_3]/I=12.
$
In particular, the singular scheme of $B$ on the chart $x_2=1$ is
zero-dimensional of length twelve.

Set $h=\det(\operatorname{Hess}(g))$. With respect to the standard
monomial basis above, the determinant of multiplication by $h$ on
$\QQ[x_0,x_1,x_3]/I$ is
$
\frac{2^{56}3^{24}13^4 31^2}{29^{14}}\,
712331^2\,540913480586273^2\neq0.
$
Hence multiplication by $h$ is an isomorphism, so $h$ is a unit in
the critical algebra. In particular, the Hessian of $g$ is
nondegenerate at every point of the singular scheme.

It follows that each critical point is reduced. Indeed, locally at
such a point the Jacobian matrix of $(g_{x_0},g_{x_1},g_{x_3})$ is the Hessian
matrix of $g$, which is invertible. Thus the three partial derivatives
generate the maximal ideal in the completed local ring. Since the
total length is twelve, there are exactly twelve distinct singular
points of $B$ on $x_2=1$.

Let $Q$ be one of these points and choose affine coordinates
$u_1,u_2,u_3$ centered at its image in $\P^3$. Since
$g(Q)=dg(Q)=0$ and $\operatorname{Hess}(g)(Q)$ is nondegenerate, the
quadratic part of the local equation
$
w^2-g(u_1,u_2,u_3)=0
$
is a nondegenerate quadratic form in the four variables
$w,u_1,u_2,u_3$. Hence the corresponding point of $X$ is an ordinary
double point. Thus $X$ has twelve ordinary double points on the chart
$x_2=1$.

We have therefore found all singular points of $X$: the point $P$ of
type $cD_4$, the two points $P_+,P_-$ of type $cA_3$, the point $P_0$
of type $cA_1$, and twelve ordinary double points, which are also of
type $cA_1$. Hence the singularity basket is
$
cD_4+2cA_3+13cA_1.
$
In particular, $\Sing(X)$ is finite. All these singularities are
isolated cDV singularities, hence they are terminal by the
Gorenstein terminal criterion
\cite[Section 2.1]{Pae24}, \cite{Rei87}.
\end{proof}

\begin{Lemma}\label{terminalYplus}
The variety $\mathcal{Y}^{+}$ is terminal along the exceptional divisor of $\varphi$. More precisely, over $P$ it has one $cA_1$ point and three quotient singularities of type $\frac12(1,1,1)$.
\end{Lemma}

\begin{proof}
We use the local equation of $X$ on $x_3=1$ and the weighted blow-up of Proposition \ref{divcontract}. On the $x_2$-chart set $x_0=t^2r$, $x_1=t^2s$, $x_2=t$, $w=t^3v$ and divide the equation by $t^6$. On the exceptional divisor $t=0$ the strict transform has equation
$v^2-4r^3+4rs^2=0$. Its derivatives with respect to $v,r,s,t$, restricted to $t=0$, are
$2v,\, -4(3r^2-s^2),\, 8rs,\, 4r.$
Hence the origin is the unique singular point of this chart on the exceptional divisor. Its quadratic part is $v^2+4rt$, of rank three, so it is of type $cA_1$.

On the $x_0$-chart of the index-one cover set $x_0=u^2$, $x_1=u^2s$, $x_2=ut$, $w=u^3v$. The residual $\mu_2$-action has weights $(1,0,1,1)$ on $(u,s,t,v)$, and the exceptional equation is $v^2+4s^2-4=0$. Its fixed locus on the hypersurface consists of $s=\pm1$, $u=t=v=0$. At these points the derivative with respect to $s$ is $8s\neq0$, so the index-one cover is smooth and $s$ can be eliminated analytically. The induced action on the tangent coordinates $(u,t,v)$ has weights $(1,1,1)$, and therefore both quotient points are of type $\frac12(1,1,1)$.

On the $x_1$-chart of the index-one cover set $x_0=u^2r$, $x_1=u^2$, $x_2=ut$, $w=u^3v$. The residual $\mu_2$-action again has weights $(1,0,1,1)$ on $(u,r,t,v)$, and the exceptional equation is $v^2-4r^3+4r=0$. Its fixed locus consists of $r=0,\pm1$, $u=t=v=0$. Since the derivative with respect to $r$ is $4-12r^2$, the index-one cover is smooth at all three points. The points $r=\pm1$ are the two quotient points already seen on the $x_0$-chart, while $r=0$ gives a third point of type $\frac12(1,1,1)$.

Finally, on the $w$-chart of the index-one cover set $x_0=u^2r$, $x_1=u^2s$, $x_2=ut$, $w=u^3$. The residual $\mu_3$-action has weights $(1,1,1,2)$ on $(u,r,s,t)$, and the exceptional equation is $1-4r^3+4rs^2=0$. Its only possible fixed point is the origin, which does not lie on the hypersurface. Hence there are no further singularities over $P$.
\end{proof}

\begin{Proposition}\label{factorialX}
The sextic double solid $X$ is a terminal factorial Fano threefold and
$\Cl(X)=\Pic(X)=\ZZ[\OO_X(1)]$. In particular, $\rho(X)=1$ and $X\rightarrow\operatorname{Spec}(\C)$ is a Mori fiber space.
\end{Proposition}

\begin{proof}
We first note that $\mathcal Y^+$ is normal. It is an effective Cartier
divisor in the Cohen--Macaulay toric variety $\mathcal P^+$, hence it
satisfies $S_2$. The small birational map
$\mathcal Y\dashrightarrow\mathcal Y^+$ is an isomorphism outside subsets
of codimension at least two, and $\mathcal Y$ is normal by
Lemma~\ref{fibersintegral}. Therefore $\mathcal Y^+$ is regular in
codimension one and satisfies $R_1$. Serre's criterion gives normality.

A small birational map between normal varieties induces an isomorphism of
class groups, hence Proposition~\ref{ClY} gives
$\Cl(\mathcal{Y}^{+})=\ZZ H\oplus\ZZ F$. The morphism $\varphi$ contracts only the prime divisor $E=(p=0)$, whose class is $H$. The standard exact sequence for class groups of a birational contraction therefore yields
$
\Cl(X)\cong(\ZZ H\oplus\ZZ F)/\ZZ H\cong\ZZ.
$
The class $L=F+2H$ defining $\varphi$ descends to $\OO_X(1)$ and generates this quotient. The only singular point of $\P(1,1,1,1,3)$ is $[0:0:0:0:1]$, and it does not belong to $X$ since the left hand side of \eqref{sexticEq} is then nonzero. Thus $\OO_X(1)$ is Cartier, and therefore
$\Cl(X)=\Pic(X)=\ZZ[\OO_X(1)]$. Hence $X$ is factorial and $\rho(X)=1$. By adjunction, $-K_X=\OO_X(1)$, so $X$ is Fano. Terminality follows from Proposition~\ref{singX}, and the last assertion follows.
\end{proof}

\begin{Proposition}\label{MFS}
The morphism $\pi:\mathcal{Y}\rightarrow\P^{1}$ is a Mori fiber space of degree one, the map $\mathcal{Y}\dashrightarrow\mathcal{Y}^{+}$ is a flop, and
$
X\longleftarrow\mathcal{Y}^{+}\dashrightarrow\mathcal{Y}
\longrightarrow\P^{1}
$
is a Sarkisov link of type I centered at the $cD_4$ point $P\in X$.
\end{Proposition}

\begin{proof}
By Proposition \ref{singX}, $X$ is terminal. By Proposition
\ref{divcontract}, the morphism
$\varphi:\mathcal{Y}^{+}\rightarrow X$ is an isomorphism away from its
exceptional divisor $E=(p=0)$. Hence $\mathcal{Y}^{+}$ is terminal away
from $E$, while Lemma \ref{terminalYplus} shows that it is terminal
along $E$. Thus $\mathcal{Y}^{+}$ is terminal.

By Proposition \ref{flop}, the map
$\chi:\mathcal{Y}\dashrightarrow\mathcal{Y}^{+}$ is small. Hence it
induces an isomorphism
$\Cl(\mathcal{Y})\simeq\Cl(\mathcal{Y}^{+})$. Proposition \ref{ClY}
gives
$
\Cl(\mathcal{Y})
=\ZZ H\oplus\ZZ F.
$
The classes $H$ and $F$ are restrictions of $\QQ$-Cartier
toric divisor classes in both GIT chambers. Therefore
$\mathcal{Y}^{+}$ is $\QQ$-factorial. Moreover,
Proposition \ref{ClY} already gives that $\mathcal{Y}$ is
$\QQ$-factorial and
$\rho(\mathcal{Y}/\P^{1})=1$.

By Proposition \ref{flop}, the small birational map $\chi$ is
crepant. Since $\mathcal{Y}^{+}$ is terminal, discrepancies on a
common resolution show that $\mathcal{Y}$ is terminal as well.
Furthermore,
$
-K_{\mathcal{Y}}=F+H
$
and $-K_{\mathcal{Y}}\equiv_{\pi}H$. The class $H$ is the tautological
relatively ample class of the weighted projective bundle
$\mathcal{P}\rightarrow\P^{1}$, so
$-K_{\mathcal{Y}}$ is $\pi$-ample. By Lemma \ref{fibersintegral}, all
fibers of $\pi$ are integral, hence connected. The generic fiber is a
del Pezzo surface of degree one by Construction~\ref{Yconstruction}.
Therefore
$
\pi:\mathcal{Y}\longrightarrow\P^{1}
$
is a Mori fiber space of degree one.

Since $\chi$ is small and crepant and both
$\mathcal{Y}$ and $\mathcal{Y}^{+}$ are terminal and
$\QQ$-factorial, $\chi$ is a flop.

By Proposition \ref{divcontract},
$\varphi:\mathcal{Y}^{+}\rightarrow X$ is a divisorial contraction
whose exceptional divisor is $E=(p=0)$ and whose image is the $cD_4$
point $P\in X$. Moreover, $\varphi$ is locally the weighted blow-up of
$P$ with weights $(2,2,1,3)$ in the coordinates $(x_0,x_1,x_2,w)$ and has
discrepancy one. Finally, Proposition
\ref{factorialX} shows that $X\rightarrow\operatorname{Spec}(\C)$
is a Mori fiber space. Hence
$
X\longleftarrow\mathcal{Y}^{+}\dashrightarrow\mathcal{Y}
\longrightarrow\P^{1}
$
is a Sarkisov link of type I centered at $P$.
\end{proof}

\begin{proof}[Proof of Theorem~\ref{intro:sextic}]
For $\lambda=2$, Propositions~\ref{singX}, \ref{MFS}, and
\ref{factorialX} give terminality, factoriality, Picard rank one, the
stated singularity basket, and the Sarkisov link. By
Proposition~\ref{prop:clop-p2}, the generic fiber $S$ of
$\mathcal{Y}\to\P^1$ admits a dominant rational map
$\P^2_{\C(t)}\dashrightarrow S$ of degree nine. Spreading it over the
base and composing with the birational map to $X_2$ gives a dominant
rational map $\P^3\dashrightarrow X_2$ of degree nine.

For arbitrary $\lambda$, replace $c=2t+1$ by
$c_\lambda=\lambda t+1$ in the cubic pencil and in
\eqref{jacobianDP1}. The invariant calculation and the toric two-ray game
are polynomial in $\lambda$ and give precisely
\eqref{eq:sextic-family}. On the chart $z=1$, elimination of $y$ from the
base scheme gives
$
P_\lambda(t,x)=(t^2-1)^3x^9+t(\lambda t+1)(t^2-1)^2x^6
-t^4(t^2-1)x-t^5(t+2).
$
Take the projective closure of $P_\lambda(t,x)=0$ in a fixed product of
projective lines. After shrinking the parameter line around $2$, this is a
flat projective family. Its fiber at $2$ is geometrically integral by
Lemma~\ref{baseirr}; openness of geometric integrality therefore gives an
open neighborhood on which every fiber is geometrically integral. Hence
$P_\lambda$ is irreducible over $\C(t)$ for every constant parameter in
this neighborhood. After a further shrinking, its degree is nine and it is
separable, so the base scheme is a single Galois orbit of nine reduced
points.

Over $K=\C(t)$, the universal family of members of the cubic pencil is a
proper flat family over $\mathbb A^1_\lambda\times\P^1_{[U:V]}$. By
Lemma~\ref{admissiblePencil} and
Lemma~\ref{lem:uniform-integrality}, after shrinking around $2$ every
geometric member of every pencil is integral. Reducedness of the base
scheme and smoothness of the generic member are open conditions. Thus all
these pencils are admissible. The resulting degree-one del Pezzo surface
$S_\lambda/K$ has Picard rank one by
\cite[Corollary~3.8]{CLOP26}. Since it is geometrically rational and
$-K_{S_\lambda}$ is primitive,
$
\Pic(S_\lambda)=\ZZ[-K_{S_\lambda}].
$
The open subset obtained over $K$ is defined by finitely many nonzero
rational functions in the constant parameter $\lambda$; excluding the
common zeros of their coefficients gives a nonempty Zariski open subset of
$\mathbb A^1_\C$ containing $2$.

The global equations define a flat family
$\mathfrak Y\to\mathbb A^1_\lambda\times\P^1$. Applying
Lemma~\ref{lem:uniform-integrality} to the fiber at $2$ and shrinking once
more, every fiber of every $\mathcal Y_\lambda\to\P^1$ is geometrically
integral. The normality argument of Lemma~\ref{fibersintegral} is uniform in
this family. The intersections of the relative GIT unstable strata with
the hypersurfaces form projective families, and fiber dimension is upper
semicontinuous. Since at $\lambda=2$ the wall crossing has only curves as
exceptional loci, it remains small after shrinking. Finally, the final
hypersurfaces form a flat Gorenstein family, and terminality is open in
such a family. We therefore obtain a nonempty Zariski open subset
$U_6\subset\mathbb A^1$ containing $2$ on which all these properties hold.

The proof of Proposition~\ref{ClY} now applies to
$\mathcal Y_\lambda$ and gives
$
\Cl(\mathcal Y_\lambda)=\ZZ H\oplus\ZZ F.
$
The small wall crossing preserves the class group. The final contraction
has exceptional divisor
$
E_\lambda=\{p=0,\ v^2=4r(r^2-s^2)z^3\},
$
which is independent of $\lambda$, prime, and has class $H$. Hence
$
\Cl(X_\lambda)\cong
\frac{\ZZ H\oplus\ZZ F}{\ZZ H}\cong\ZZ.
$
The image of $L=F+2H$ is $\OO_{X_\lambda}(1)$ and generates this group.
The hypersurface avoids the singular point of $\P(1,1,1,1,3)$, so this
line bundle is Cartier. Consequently
$
\Cl(X_\lambda)=\Pic(X_\lambda)
=\ZZ[\OO_{X_\lambda}(1)],
$
and $X_\lambda$ is factorial with $\rho(X_\lambda)=1$. By adjunction it
is Fano. The CLOP map on the generic fiber has degree nine for every
$\lambda\in U_6$; spreading it over the base gives a dominant rational map
$\P^3\dashrightarrow X_\lambda$ of degree nine. This proves
unirationality and the asserted degree-one del Pezzo fibration.
\end{proof}

\section{Quartic threefolds}\label{quarticSec}
In this section we carry out an analogous construction for quartic
threefolds. Starting from an explicit one-parameter family of degree-two
del Pezzo surfaces over $\C(t)$, we compute the associated relative
Jacobians and globalize them to obtain a family of quartic threefolds.
We analyze in detail the distinguished member corresponding to
$\kappa=1$, proving that it is terminal, factorial, and unirational, and
then extend these properties to a nonempty open subset of the family.

Fix $\iota\in\C$ such that $\iota^2=-3$. For $\kappa\in\mathbb A^1$ set
\begin{equation}\label{eq:quartic-source-family}
Z_\kappa=\left\{
\omega^2=t y^4+\frac{t}{3}x^2y^2-\frac{1}{t}x^3y
-\kappa y^3z+(t^2y^2-x^2-xy)z^2+t z^4
\right\}
\subset\P_{\C(t)}(1_x,1_y,1_z,2_\omega).
\end{equation}
The invariant calculation below gives the family $X_{4,\kappa}$ of
\eqref{eq:quartic-family}. We first treat $\kappa=1$. Set $K=\C(t)$ and
write
$$
Z=\left\{
\omega^2=t y^4+\frac{t}{3}x^2y^2-\frac{1}{t}x^3y-y^3z
+(t^2y^2-x^2-xy)z^2+t z^4
\right\}
\subset\P_K(1_x,1_y,1_z,2_\omega).
$$

\begin{Lemma}\label{quarticSourceSmooth}
The surface $Z$ is a smooth del Pezzo surface of degree two.
\end{Lemma}

\begin{proof}
It is enough to prove that the branch quartic is smooth over $K$.
Specializing $t=3$ gives
$
3y^4+x^2y^2-\frac13x^3y-y^3z+(9y^2-x^2-xy)z^2+3z^4.
$
On each of the affine charts $x=1$, $y=1$, and $z=1$, exact
Gr\"obner-basis reduction shows that the ideal generated by this equation
and its first derivatives in the two affine variables is the unit ideal.
Thus the specialized quartic is smooth. The discriminant of the branch
quartic over $\C(t)$ is therefore nonzero, so the generic branch quartic is
smooth.

The hypersurface $Z$ does not contain the singular point
$[0:0:0:1]$ of $\P(1,1,1,2)$. Hence $Z$ is smooth. Adjunction gives
$-K_Z=\OO_Z(1)$, and, if $H=\OO_Z(1)$, then
$
H^2=\frac{4}{1\cdot1\cdot1\cdot2}=2.
$
Therefore $Z$ is a del Pezzo surface of degree two.
\end{proof}

Consider the anticanonical pencil induced by the lines through
$P_0=[0:0:1]\in\P^2_K$. Writing such a line as $px+qy=0$ and setting
$x=qu$, $y=-pu$, $z=v$, its inverse image is
\begin{equation}\label{quarticBinary}
\omega^2=a u^4+p^3u^3v+c u^2v^2+t v^4,
\end{equation}
where
$
a=tp^4+\frac{t}{3}p^2q^2+\frac{1}{t}pq^3,
\, c=t^2p^2-q^2+pq.
$
The base locus is the degree-two closed point over $P_0$ defined by
$\omega^2=t$.

\begin{Lemma}\label{quarticAdmissible}
The pencil above is admissible.
\end{Lemma}

\begin{proof}
The base scheme is reduced since $t$ is not a square in $K$. To prove
geometric irreducibility, use \eqref{quarticBinary}. If $p\neq0$, the
coefficient of $u^3v$ is nonzero. A binary quartic
$
a u^4+b u^3v+c u^2v^2+e v^4
$
with $b\neq0$, no $uv^3$ term, and $e\neq0$ cannot be a square: if it
were $(\alpha u^2+\beta uv+\gamma v^2)^2$, then the vanishing of the
$uv^3$ coefficient and $\gamma\neq0$ would give $\beta=0$, contradicting
$b\neq0$. If $p=0$, the right-hand side is
$v^2(tv^2-q^2u^2)$, which is also not a square. Hence every geometric
member is integral.

The generic member is smooth because the specialization $p=1$, $q=0$,
$t=3$ gives the binary quartic
$
3u^4+u^3v+9u^2v^2+3v^4,
$
whose discriminant is $317601\neq0$. Thus the pencil is admissible.
\end{proof}

Let $S/K$ be the degree-one del Pezzo surface obtained from this pencil by
the construction of \cite[Section~3]{CLOP26}. Taking $A_Z=-K_Z$ in
\cite[Lemma~3.4]{CLOP26} gives a dominant rational map
\begin{equation}\label{quarticDegree4}
Z\dashrightarrow S
\end{equation}
of degree four.

Set
$
L=\iota\bigl(-18t^2p^2q-36pq^2+6q^3\bigr).
$
Let $X_W,Y_W$ denote the Weierstrass coordinates of the Jacobian of
\eqref{quarticBinary}. Computing its invariants and making the changes
$X_W=z+Q$ and $Y_W=w+L$, where
$Q=3t^2p^2+18pq+3q^2$, gives
\begin{equation}\label{quarticDP1eq}
\begin{aligned}
w^2+2Lw={}&z^3+(9t^2p^2+54pq+9q^2)z^2
-27p(12p^3t^2-10p^2qt^2-35pq^2-2q^3)z\\
&+(729t-2916t^4)p^6-7776t^2p^5q+2997t^2p^4q^2
+5400p^3q^3-243p^2q^4.
\end{aligned}
\end{equation}

Let $\mathcal T$ be the simplicial toric variety with Cox ring
$\C[r,s,p,q,z,w]$, irrelevant ideal $(r,s)\cap(p,q,z,w)$, and grading
$$
\begin{pmatrix}
1&1&0&1&1&1\\
0&0&1&1&2&3
\end{pmatrix},
$$
where the rows are denoted by $F,H$. Thus
$\mathcal T\to\P^1_{[r:s]}$ is a $\P(1,1,2,3)$-bundle. Set
$
\mathcal L=\iota(-18s^2p^2q-36rpq^2+6q^3).
$

\begin{Construction}\label{quarticGlobalModel}
Let $\mathcal V\subset\mathcal T$ be the hypersurface
\begin{equation}\label{quarticGlobal}
\begin{aligned}
r^2w^2+2\mathcal Lw={}&rz^3+(9s^2p^2+54rpq+9q^2)z^2-27p(12rs^2p^3-10s^2p^2q-35rpq^2-2q^3)z\\
&+(729sr^3-2916s^4)p^6-7776rs^2p^5q
+2997s^2p^4q^2+5400rp^3q^3-243p^2q^4.
\end{aligned}
\end{equation}
Every term has class $4F+6H$. Thus $\mathcal V$ is a relative Cartier
divisor. It contains no fiber of $\mathcal T\to\P^1$, so
$\mathcal V\to\P^1$ is flat. On the chart $r=1$, with $t=s$, its generic
fiber is \eqref{quarticDP1eq}, and hence is $S$.
\end{Construction}
Set
$A=F+H,\, B=F+2H,\, M=F+3H.$ Let $\mathcal T_A$ be the model corresponding to the chamber
$\langle A,B\rangle$, with irrelevant ideal
$(r,s,q)\cap(z,w,p)$, and let $\mathcal V_A$ be the transform of
$\mathcal V$. Since $4F+6H=2A+2B$, the hypersurface $\mathcal V_A$ is
ample in the projective simplicial toric fourfold $\mathcal T_A$.

\begin{Lemma}\label{quarticH2VA}
Restriction induces an isomorphism
$
H^2(\mathcal T_A,\QQ)\xrightarrow{\sim}H^2(\mathcal V_A,\QQ).
$
In particular, $H^2(\mathcal V_A,\QQ)=\QQ F\oplus\QQ H$.
\end{Lemma}

\begin{proof}
Exact Jacobian calculations on local uniformizing covers show that the
singularities of $\mathcal V_A$ not inherited from the simplicial toric
ambient space are isolated hypersurface singularities. On $p\neq0$ there
are two such points, corresponding to the points $Q_\pm$ in
Proposition~\ref{quarticBasket}. On $p=0$ there is one further orbit,
$
r=q=z=p=0,\, sw\neq0.
$
On the uniformizing chart $s=\widetilde w=1$, where
$\widetilde w=\iota w$, its quadratic and cubic terms are
$
r^2+108p^2q-36q^3,
$
so it is an isolated compound $D_4$ hypersurface singularity.

Choose a general hypersurface $\mathcal V'\in|2A+2B|$. It is
quasi-smooth. Since $2A+2B$ is ample, the toric Lefschetz theorem
\cite[Proposition~10.8]{BC94} gives
$
H^2(\mathcal T_A,\QQ)\xrightarrow{\sim}H^2(\mathcal V',\QQ).
$
A deformation from $\mathcal V_A$ to $\mathcal V'$ smooths only finitely
many isolated hypersurface singularities on the local uniformizing covers.
The Milnor fiber of a three-dimensional isolated hypersurface singularity
has the homotopy type of a bouquet of three-spheres
\cite[Theorem~6.5]{Mil68}. Mayer--Vietoris therefore shows that replacing
the cone on the link by the Milnor fiber does not change $H^2$. Taking
invariants under the finite local quotient groups is exact over $\QQ$, so
the same conclusion holds on the toric orbifold. Hence
$
H^2(\mathcal V_A,\QQ)\simeq H^2(\mathcal V',\QQ).
$
Finally, $\mathcal T_A$ has Picard number two and its degree-two cohomology
is generated by $F$ and $H$.
\end{proof}

The rational map $[r:s]$ on $\mathcal V_A$ has base curve
$
C_+=\{r=s=0\}\cap\mathcal V_A\simeq\P^1.
$
Indeed, stability gives $q\neq0$, and after setting $q=1$ its equation is
$
4\iota w=3z^2+18pz-81p^2.
$
The nonzero derivative with respect to $w$ also shows that $C_+$ is
contained in the smooth locus of $\mathcal V_A$. Let
$
\beta:\mathcal W=\Bl_{C_+}(\mathcal V_A)\longrightarrow\mathcal V_A
$
be the blow-up, with exceptional divisor $E$. The Rees construction for
the ideal $(r,s)$ resolves the rational map $[r:s]$ and gives a morphism
$\pi:\mathcal W\to\P^1$. This is the common resolution of the first wall
crossing; on the original side it is the blow-up of
$
C_-=\{p=z=w=0\}\cap\mathcal V\simeq\P^1.
$
Consequently the generic fiber of $\pi$ is
\begin{equation}\label{quarticBlownFiber}
\widetilde S=\Bl_P(S),\qquad P=[0:1:0:0]\in S(K).
\end{equation}

\begin{Lemma}\label{quarticPicardZ}
One has $\rho(S)=\rho(Z)=1$.
\end{Lemma}

\begin{proof}
Since $C_+$ is a smooth curve in the smooth locus, the blow-up formula and
Lemma~\ref{quarticH2VA} give
$
H^2(\mathcal W,\QQ)=\QQ F\oplus\QQ H\oplus\QQ E.
$
Both $r$ and $s$ vanish to first order along $C_+$, so
$
\pi^*\OO_{\P^1}(1)=F-E.
$
Therefore $F|_{\widetilde S}=E|_{\widetilde S}$, and the image of
$H^2(\mathcal W,\QQ)\to H^2(\widetilde S,\QQ)$ is generated by
$H|_{\widetilde S}$ and $E|_{\widetilde S}$.

Choose a resolution
$\widetilde{\mathcal W}\to\mathcal W$ which is an isomorphism above an
open subset of the base on which $\pi$ is smooth. It may be chosen as a
sequence of blow-ups along smooth centers contained in the singular fibers.
The blow-up formula expresses $H^2(\widetilde{\mathcal W},\QQ)$ as the
pullback of $H^2(\mathcal W,\QQ)$ together with the classes of the
exceptional divisors. All these exceptional divisors are vertical and
restrict trivially to the generic fiber. Hence the image of
$
H^2(\widetilde{\mathcal W},\QQ)\longrightarrow
H^2(\widetilde S,\QQ)
$
is still generated by $H|_{\widetilde S}$ and $E|_{\widetilde S}$.
Applying the global invariant cycle theorem to the smooth projective family
over the smooth locus of the base shows that this image is the
monodromy-invariant part of $H^2(\widetilde S,\QQ)$
\cite[Corollary~1.40]{PS08}. Every divisor class defined over $K$ is
monodromy invariant. Thus $\rho(\widetilde S)\leq2$. The classes
$H|_{\widetilde S}$, the pullback of $-K_S$, and
$E|_{\widetilde S}$, the exceptional curve, are independent and defined
over $K$, so $\rho(\widetilde S)=2$.

By \eqref{quarticBlownFiber},
$\rho(\widetilde S)=\rho(S)+1$, and hence $\rho(S)=1$. Finally, the base
scheme of the admissible pencil on $Z$ is one Galois orbit. Therefore
\cite[Corollary~3.7]{CLOP26} gives
$
\rho(S)=\rho(Z)+1-1=\rho(Z),
$
so $\rho(Z)=1$.
\end{proof}

\begin{Lemma}\label{quarticClV}
Every fiber of $\mathcal V\to\P^1$ is geometrically integral,
$\mathcal V$ is normal, and
$
\Cl(\mathcal V)=\ZZ H\oplus\ZZ F.
$
In particular, $\mathcal V$ is $\QQ$-factorial and
$\rho(\mathcal V/\P^1)=1$.
\end{Lemma}

\begin{proof}
For $r\neq0$, completing the square in $w$ gives a double-cover equation
whose branch form has a nonzero $z^3$ term. A weighted form of degree three
in $p,q,z$, with $\deg p=\deg q=1$ and $\deg z=2$, is at most linear in
$z$, so its square has no $z^3$ term. Thus the branch form is not a square,
and the fiber is geometrically integral.

For $r=0$, equation \eqref{quarticGlobal} is linear in $w$, with
coefficient
$
12\iota q(q^2-3s^2p^2),
$
which is coprime to the remaining term. Hence this fiber is geometrically
integral as well.

The toric variety $\mathcal T$ is Cohen--Macaulay and $\mathcal V$ is an
effective Cartier divisor, so $\mathcal V$ satisfies $S_2$. The generic
fiber is smooth, and at the generic point of a special integral fiber the
one-dimensional local ring is regular by flatness over the smooth base
curve. Thus $\mathcal V$ satisfies $R_1$ and is normal.

Since $S$ is geometrically rational, $\Pic(S)$ is torsion-free. Moreover,
$-K_S$ is primitive because $K_S^2=1$. Lemma~\ref{quarticPicardZ}
therefore gives
$
\Pic(S)=\ZZ[-K_S]=\ZZ[H|_S].
$
Let $D$ be a prime Weil divisor on $\mathcal V$. There is an integer $m$
such that $D|_S\sim mH|_S$. After subtracting $mH$ and a principal
divisor, we may assume that $D$ is vertical. Every vertical prime divisor
is a fiber, since all fibers are integral, and hence has class $F$. Thus
$H,F$ generate $\Cl(\mathcal V)$. They are independent: $H$ has positive
degree on curves in a general fiber, whereas $F$ has positive degree on a
curve dominating the base. Both are restrictions of toric
$\QQ$-Cartier classes, which proves the last assertions.
\end{proof}

The rays of the Cox coordinates are ordered as
$
F,\, A,\, B,\, M,\, H.
$
Let $\mathcal T_B$ be the model corresponding to the chamber
$\langle B,M\rangle$, with irrelevant ideal
$(r,s,q,z)\cap(w,p)$, and let $\mathcal V_B$ be the transform of
$\mathcal V_A$.

\begin{Proposition}\label{quarticWalls}
The wall crossings
$
\mathcal V\dashrightarrow\mathcal V_A\dashrightarrow\mathcal V_B
$
are small. Their exceptional loci on the two sides are the curves
$$
\begin{aligned}
C_-&=\{p=z=w=0\}\subset\mathcal V,
& C_+&=\{r=s=0\}\subset\mathcal V_A,\\
D_-&=\{p=w=0,\ rz+9q^2=0\}\subset\mathcal V_A,
& D_+&=\{r=s=q=0\}\subset\mathcal V_B.
\end{aligned}
$$
All four are rational. In particular,
$
\Cl(\mathcal V_B)=\Cl(\mathcal V)=\ZZ H\oplus\ZZ F.
$
\end{Proposition}

\begin{proof}
For the first wall, the locus which is stable in the original chamber and
unstable in the chamber $\langle A,B\rangle$ is $p=z=w=0$.
Stability gives $q\neq0$ and $(r,s)\neq(0,0)$, so its intersection with
$\mathcal V$ is $C_-\simeq\P^1$. In the opposite direction the new locus
is $r=s=0$. Stability gives $q\neq0$ and $(p,z,w)\neq(0,0)$. Setting
$q=1$, equation \eqref{quarticGlobal} becomes
$
4\iota w=3z^2+18pz-81p^2,
$
so the quotient is $C_+\simeq\P^1$. Hence the first wall crossing is
small.

For the second wall, the locus lost from $\mathcal T_A$ is $p=w=0$.
Stability gives $z\neq0$, and equation \eqref{quarticGlobal} reduces to
$
z^2(rz+9q^2)=0.
$
Thus its intersection is $D_-$; after setting $z=1$, its quotient is
$\P(1,2)$ and hence is a rational curve. On the other side the new locus is
$r=s=q=0$. Stability gives $z\neq0$ and $(p,w)\neq(0,0)$, while the
hypersurface equation vanishes identically there. After setting $z=1$, the
residual weights of $p$ and $w$ are both one, so
$D_+\simeq\P^1$. The second wall crossing is small.

The transforms are effective Cartier divisors in Cohen--Macaulay toric
varieties, hence satisfy $S_2$. They are isomorphic to the normal variety
$\mathcal V$ outside subsets of codimension at least two, so they satisfy
$R_1$ and are normal. A small birational map between normal varieties
induces an isomorphism of class groups. Lemma~\ref{quarticClV} gives the
last assertion.
\end{proof}

We write $\mathcal V^+=\mathcal V_B$.

\begin{Proposition}\label{quarticContraction}
The toric variety $\mathcal T_B$ is the weighted blow-up
$
\Phi:\mathcal T_B\longrightarrow\P^4
$
of $P=[0:0:0:0:1]$ with weights $(3,3,2,1)$ in the affine coordinates
$(x_0,x_1,x_2,x_3)$. It is given by
\begin{equation}\label{quarticAmbientContraction}
[x_0:x_1:x_2:x_3:x_4]
=[rp^3:sp^3:qp^2:zp:w].
\end{equation}
Its restriction
$
\varphi:\mathcal V^+\longrightarrow X_4=X_{4,1}
$
is a divisorial contraction. Its exceptional divisor is
$
E=\{p=0,\ r^2w^2+12\iota q^3w-rz^3-9q^2z^2=0\},
$
and $\varphi(E)=P$.
\end{Proposition}

\begin{proof}
In the basis $(M,H)$ the degrees of $r,s,q,z,w,p$ are
$
(1,-3),\ (1,-3),\ (1,-2),\ (1,-1),\ (1,0),\ (0,1).
$
The monomials in \eqref{quarticAmbientContraction} have degrees
$(1,0),(1,0),(1,0),(1,0),(1,0)$ and generate the section ring of $M$.
They therefore define $\Phi$. Away from $p=0$, the second torus action
sets $p=1$ and recovers $r,s,q,z,w$, so $\Phi$ is an isomorphism. On
$p=0$, stability gives $w\neq0$, and the whole divisor is contracted to
$P$. On the chart $x_4=1$, the valuation of $p=0$ on
$(x_0,x_1,x_2,x_3)$ is $(3,3,2,1)$; hence $\Phi$ is the stated weighted
blow-up.

Let $f_1$ be the quartic polynomial in \eqref{eq:quartic-family} with
$\kappa=1$, and let $\mathcal F$ be the left-hand side minus the right-hand
side of \eqref{quarticGlobal}. Direct substitution gives the exact identity
\begin{equation}\label{quarticStrictTransform}
f_1(rp^3,sp^3,qp^2,zp,w)=p^6\mathcal F(r,s,p,q,z,w).
\end{equation}
Thus $\mathcal V^+$ is the strict transform of $X_4$ and
$\varphi=\Phi|_{\mathcal V^+}$. Setting $p=0$ gives the displayed equation
of $E$. As a quadratic polynomial in $w$, its discriminant is
$
4(r^3z^3+9r^2q^2z^2-108q^6)=4(rz-3q^2)(rz+6q^2)^2,
$
which is not a square in $\C(r,q,z)$. Hence $E$ is prime and the
restriction is divisorial.
\end{proof}

\begin{Proposition}\label{quarticBasket}
The singularity basket of $X_4$ is
$
cD_4+2cA_2.
$
In particular, $X_4$ is terminal and normal.
\end{Proposition}

\begin{proof}
Set $\widetilde x_4=\iota x_4$, so that, after multiplication by a nonzero
scalar, the equation has rational coefficients. At
$P=[0:0:0:0:1]$, the quadratic part on the chart
$\widetilde x_4=1$ is $x_0^2$. After splitting the nondegenerate
$x_0$-direction, the cubic part of the residual function is
$
-36x_2(x_2^2-3x_1^2),
$
a product of three distinct linear forms. Hence $P$ is of type $cD_4$. A standard computation on $x_0=0$ gives, besides $P$, exactly
the two points
$
Q_\pm=\left[0:\pm\frac1{\sqrt3}:1:-6:\frac{27}{4}\right]
$
in the coordinates
$[x_0:x_1:x_2:x_3:\widetilde x_4]$. At each point the Hessian has rank three. After choosing a generator of
its one-dimensional kernel, the cubic directional coefficient is nonzero
(an exact computation over $\QQ(\sqrt3)$ verifies this). By the analytic
splitting lemma, both points are of type $cA_2$.

There are no further singularities. On the chart $x_0=1$, the derivative
with respect to $\widetilde x_4$ eliminates this variable; the exact
Gr\"obner basis of the resulting Jacobian ideal in
$\QQ[x_1,x_2,x_3]$ is $\{1\}$. The charts on
$x_0=\widetilde x_4=0$ give the unit ideal as well. Therefore
$
\Sing(X_4)=\{P,Q_+,Q_-\}.
$
All three points are isolated cDV singularities. An isolated Gorenstein
threefold singularity is terminal precisely when it is cDV
\cite[Section~2.1]{Pae24}, \cite{Rei87}; hence $X_4$ is terminal.
\end{proof}

\begin{Proposition}\label{quarticFactorial}
The quartic $X_4$ is factorial and
$
\Cl(X_4)=\Pic(X_4)=\ZZ[\OO_{X_4}(1)].
$
\end{Proposition}

\begin{proof}
By Proposition~\ref{quarticWalls},
$\Cl(\mathcal V^+)=\ZZ H\oplus\ZZ F$. Proposition~\ref{quarticContraction}
shows that the only contracted prime divisor is $E=(p=0)$, whose class is
$H$. Since $X_4$ is normal by Proposition~\ref{quarticBasket}, the exact
sequence for class groups of a birational contraction gives
$
\Cl(X_4)\simeq
\frac{\ZZ H\oplus\ZZ F}{\ZZ H}\simeq\ZZ.
$
The class $M=F+3H$ defines the contraction and descends to
$\OO_{X_4}(1)$. Modulo $H$ it agrees with $F$, so it generates
$\Cl(X_4)$. Since $\OO_{X_4}(1)$ is Cartier, $X_4$ is factorial and the
displayed equality follows.
\end{proof}

\begin{proof}[Proof of Theorem~\ref{intro:quartic}]
For $\kappa=1$, Propositions~\ref{quarticBasket} and
\ref{quarticFactorial} give terminality, the singularity basket,
factoriality, and Picard rank one. By
Proposition~\ref{prop:dp2-unirational}, the degree-two del Pezzo surface
$Z$ is $K$-unirational. Composing a dominant rational map
$\P^2_K\dashrightarrow Z$ with \eqref{quarticDegree4} shows that $S$ is
$K$-unirational. Spreading over the base gives a dominant rational map
$\P^2\times\P^1\dashrightarrow\mathcal V$, so $X_{4,1}$ is unirational.

For arbitrary $\kappa$, the pencil of lines through $[0:0:1]$ on
\eqref{eq:quartic-source-family} has binary quartic
$
\omega^2=a u^4+\kappa p^3u^3v+c u^2v^2+t v^4,
$
with the same $a,c$ as in \eqref{quarticBinary}. The invariant calculation
differs only by the term $-27\kappa^2p^6t$ in the cubic invariant. The same
changes of variables and toric contraction replace
$-729x_0^3x_1$ by $-729\kappa^2x_0^3x_1$, giving
\eqref{eq:quartic-family}.

Smoothness of $Z_\kappa$ is open in $\kappa$. Over $K=\C(t)$, apply
Lemma~\ref{lem:uniform-integrality} to the universal family of members of
the anticanonical pencil over the parameter line and its proper pencil
line. Together with openness of reducedness of the base scheme and
smoothness of the generic member, this gives, after shrinking around $1$,
a smooth degree-two del Pezzo surface with an admissible pencil for every
constant parameter under consideration. As in the sextic case, the open
conditions obtained over $K$ exclude only a proper closed set of constant
parameters.

The global equations form a flat family over
$\mathbb A^1_\kappa\times\P^1$. Lemma~\ref{lem:uniform-integrality},
applied to $\kappa=1$, makes every fiber of every global model
geometrically integral after shrinking. Isolatedness of the hypersurface
singularities on the local uniformizing covers and smoothness of the first
wall center are open. Upper semicontinuity of the dimensions of the
relative exceptional loci keeps both wall crossings small. Terminality of
the final quartic is open in this flat Gorenstein family. We obtain a
nonempty Zariski open subset $U_4\subset\mathbb A^1$ containing $1$ on
which all these properties hold.

For every $\kappa\in U_4$, the transform $\mathcal V_{A,\kappa}$ is an
ample hypersurface of class $2A+2B$ in the same simplicial toric fourfold.
The Lefschetz and smoothing argument of Lemma~\ref{quarticH2VA} gives
$
H^2(\mathcal V_{A,\kappa},\QQ)=\QQ F\oplus\QQ H.
$
Resolving the first wall crossing by blowing up its smooth rational center
and applying the argument of Lemma~\ref{quarticPicardZ} gives
$\rho(S_\kappa)=1$. Since $S_\kappa$ is geometrically rational and
$-K_{S_\kappa}$ is primitive,
$
\Pic(S_\kappa)=\ZZ[-K_{S_\kappa}].
$
The proof of Lemma~\ref{quarticClV} then gives
$
\Cl(\mathcal V_\kappa)=\ZZ H\oplus\ZZ F,
$
and the two small wall crossings preserve this class group.

The exceptional divisor of the final contraction is independent of
$\kappa$, since the parameter-dependent term is divisible by $p^6$ after
removing the common factor in \eqref{quarticStrictTransform}. Thus
$
E_\kappa=\{p=0,\ r^2w^2+12\iota q^3w-rz^3-9q^2z^2=0\}.
$
Its discriminant is
$4(rz-3q^2)(rz+6q^2)^2$, so it is prime. Therefore
$
\Cl(X_{4,\kappa})\cong
\frac{\ZZ H\oplus\ZZ F}{\ZZ H}\cong\ZZ.
$
The class $M=F+3H$ descends to the Cartier generator
$\OO_{X_{4,\kappa}}(1)$, and hence
$
\Cl(X_{4,\kappa})=\Pic(X_{4,\kappa})
=\ZZ[\OO_{X_{4,\kappa}}(1)].
$
Thus $X_{4,\kappa}$ is factorial, $\rho(X_{4,\kappa})=1$, and it is Fano
by adjunction. Finally, Proposition~\ref{prop:dp2-unirational} shows that
$Z_\kappa$ is $K$-unirational. Composing with the degree-four CLOP map and
spreading over the base proves that $X_{4,\kappa}$ is unirational and
birational to the asserted degree-one del Pezzo fibration.
\end{proof}

\section{Smooth double covers of quadrics}\label{sec:double-quadric}

We conclude by giving a relative Abel--Jacobi interpretation of the Roth--Iskovskikh construction for smooth double covers of a quadric threefold \cite{Roth50,Isk80}. The construction in \cite[Chapter~III, Section~2.3]{Isk80} associates to a point $x$ and a line through its image a second point on the corresponding genus-one curve. The following argument identifies this correspondence with the relative Abel--Jacobi map.

\begin{Lemma}\label{lem:correspondence-growth}
Let $X$ be an irreducible complex threefold, let $p:W\to X$ be a
morphism which, over a dense open subset of $X$, is birational over $X$ to
the projection $X\times\P^1\to X$, and let $r:W\dashrightarrow X$ be
dominant. For general $x\in X$, set $C_x=\overline{r(p^{-1}(x))}$. Assume that $C_x$ is a nonconstant rational curve containing $x$ and that
\begin{equation}\label{eq:no-invariant}
\{g\in\C(X)\mid p^*g=r^*g\}=\C.
\end{equation}
Then $X$ is unirational.
\end{Lemma}

\begin{proof}
For general $x\in X$, set $V_0(x)=\{x\}$ and
$V_{i+1}(x)=\overline{r(p^{-1}(V_i(x)))}$. Since $y\in C_y$ for general $y$, we have $V_i(x)\subset V_{i+1}(x)$.

Suppose that $\dim V_i(x)=\dim V_{i+1}(x)<3$. Then $V_i(x)=V_{i+1}(x)$. After restricting to a dense open subset of $X$, the cycles $V_i(x)$ define a rational map $\psi_i:X\dashrightarrow\operatorname{Chow}(X)$. If $y\in C_x$ is general, every point reachable from $y$ in $i$ steps is reachable from $x$ in $i+1$ steps, so $V_i(y)\subset V_{i+1}(x)=V_i(x)$. The two varieties have the same dimension, hence $V_i(y)=V_i(x)$. Thus $\psi_i$ is constant along $C_x$. It is not constant on $X$, since $x\in V_i(x)$ and $\dim V_i(x)<3$. A nonconstant rational function on its image gives $g\in\C(X)\setminus\C$ with $p^*g=r^*g$, contradicting \eqref{eq:no-invariant}. Therefore
$
\dim V_1(x)=1,\, \dim V_2(x)=2,\, V_3(x)=X.
$
The curve $V_1(x)$ is rational. By the assumption on $p$,
$p^{-1}(V_1(x))$ is rational, so $V_2(x)$ is a unirational surface and
hence rational. Again $p^{-1}(V_2(x))$ is rational, and it dominates $X$
through $r$.
\end{proof}

\begin{proof}[Proof of Theorem~\ref{intro:double-quadric}]
Set $H=-K_X=\pi^*\OO_Q(1)$ and let $M=F_1(Q)\simeq\P^3$. Consider
$
W=\{(z,\ell)\in X\times M\mid\pi(z)\in\ell\},
$
with projections $p:W\to X$ and $f:W\to M$. The incidence variety of
points and lines on $Q$ is a projective-line bundle over $Q$; hence its
pullback $p:W\to X$ satisfies the hypothesis of
Lemma~\ref{lem:correspondence-growth}. For a general $\ell\in M$, the
fiber $E_\ell=f^{-1}(\ell)=\pi^{-1}(\ell)$ is a smooth genus-one curve and
$\deg H|_{E_\ell}=2$. On the generic fiber of $f$, apply the Abel--Jacobi construction to $A=2H|_{E_\ell}$. Thus
$
\Phi_A:E_\ell\longrightarrow\Pic^0(E_\ell),\,
z\longmapsto\OO_{E_\ell}(4z)\otimes A^{-1}.
$
Using the canonical action of $\Pic^0(E_\ell)$ on $\Pic^1(E_\ell)\simeq E_\ell$, set $\mu(z)=(-\Phi_A(z))\cdot z$. Then
\begin{equation}\label{eq:double-quadric-AJ}
3z+\mu(z)\sim2H|_{E_\ell}.
\end{equation}
After choosing an origin, $\mu$ is multiplication by $-3$ followed by a translation, hence has degree nine. It globalizes to a dominant rational map $\mu:W\dashrightarrow W$ over $M$. Set $r=p\circ\mu$.

We claim that \eqref{eq:no-invariant} holds. If $p^*g=r^*g$ and $h=p^*g$, then $\mu^*h=h$. Restricting to the generic fiber of $f$, a nonconstant $h$ would satisfy $\deg(h\circ\mu)=9\deg h$, a contradiction. Hence $h=f^*b$ for some $b\in\C(M)$. It follows that $b$ is constant on the curve $M_q$ of lines through every general $q\in Q$. Given two general lines $\ell_1,\ell_2\subset Q$, the quadric surface $Q\cap\langle\ell_1,\ell_2\rangle$ contains a line meeting both. Thus $b(\ell_1)=b(\ell_2)$, so $b$ and $g$ are constant.

It remains to study $C_x$ for general $x\in X$. Set $q=\pi(x)$ and let $M_q\simeq\P^1$ be the lines through $q$. A smooth model of $f^{-1}(M_q)$ is the double cover $\rho:T_x\to\mathbb F_2$ obtained by resolving the quadric cone $Q\cap T_qQ$. If $e$ and $f$ are the negative section and a fiber of $\mathbb F_2$, the branch divisor has class
$
4(e+2f)=-2K_{\mathbb F_2}.
$
For general $q$ it is smooth, so $T_x$ is a K3 surface. Since $q\notin B$, the branch divisor is disjoint from $e$ and $\rho^{-1}(e)=O\sqcup P$, where $O$ and $P$ correspond to $x$ and its conjugate point.

For general $q$, every fiber of $T_x\to\P^1$ is irreducible. Indeed, a reducible fiber contains a component mapped isomorphically onto a line of $Q$, hence a line on $X$. The lines on a smooth double quadric form a one-dimensional family \cite[Chapter~II, Proposition~3.4]{Isk80}, so their images do not cover $Q$.

Use $O$ as zero section. Since $T_x$ is a K3 surface, $\chi(\OO_{T_x})=2$, while $O\cdot P=0$. Shioda's height formula and the irreducibility of the fibers give
$
\langle P,P\rangle=4,
\,
\langle2P,2P\rangle=16=4+2(2P\cdot O),
$
so $2P\cdot O=6$. On a fiber $E_\ell$ we have $H|_{E_\ell}\sim O+P$. Setting $z=O$
in \eqref{eq:double-quadric-AJ} gives $\mu(O)=2P$. Thus $C_x$ is the
image in $X$ of the section $2P$.

It remains to prove that this image is not a point. Positivity of the
height alone is not sufficient, since the section $P$ also has positive
height but is contracted to the point conjugate to $x$. Suppose that a
section $R$ of $T_x\to M_q$ is contracted to a point $y\in X$. Then
$\pi(y)$ belongs to every line through $q$, and hence $\pi(y)=q$.
Therefore $y$ is either $x$ or its conjugate point, and $R$ is either $O$
or $P$. On the other hand,
$
\langle O,O\rangle=0,\,
\langle P,P\rangle=4,\,
\langle2P,2P\rangle=16.
$
Thus $2P$ is neither $O$ nor $P$, so its image is a nonconstant rational
curve. Moreover, $2P\cdot O=6$ shows that $2P$
meets $O$. Since $O$ is contracted to $x$, the image of $2P$ contains
$x$. Lemma~\ref{lem:correspondence-growth} now applies.
\end{proof}

\begin{Remark}
For $m\geq2$, applying the same construction to $A=mH|_{E_\ell}$ gives the fiberwise map characterized by
$
(2m-1)z+\mu_m(z)\sim mH|_{E_\ell},
$
of degree $(2m-1)^2$. This is precisely the point $x'$ in the Roth--Iskovskikh construction \cite[Chapter~III, Section~2.3]{Isk80}; the proof above uses $m=2$.
\end{Remark}

\section{Computational verification}\label{appendix:magma}

The computational supplement consists of the files
\path{Parametrize_smooth_V1.m}, \path{X6_random.m}, and
\path{Exact_certificates.py}. The source archive distributed with this
manuscript contains the versions used here; the public repository is
\begin{center}
	\begin{small}
	\url{https://github.com/msslxa/Unirational_Threefolds}
	\end{small}
\end{center}
The first script uses Magma \cite{BosmaCannonPlayoust97} and performs exact
symbolic computations over $\QQ$. It verifies the smoothness of the sextic
threefold \eqref{eq:explicit-X}, implements the two successive Manin
constructions of Section~\ref{sec:computations}, checks the identities on
the auxiliary degree-two del Pezzo surface \eqref{eq:auxiliary}, computes
the Jacobian of the genus-one pencil, and constructs the rational map
\eqref{eq:parametrization}. It also verifies its dominance.

The second Magma script gives an independent check of the explicit
parametrization. It evaluates the formulas at rational triples $(t,m,r)$
for which all denominators are nonzero, verifies that the resulting points
lie on \eqref{eq:explicit-X}, and computes the differential using dual
numbers. By default it performs $10^4$ valid tests with parameters of
absolute value at most $100$. These random tests are supplementary; the
proofs in the text use exact identities and an exact nonzero Jacobian
minor.

The Python file \path{Exact_certificates.py} contains only exact integer,
rational, and finite-field computations. It verifies the rank-$51$
tangent-space certificate in Proposition~\ref{prop:family-dimension}, the
resultant and squarefreeness calculation in Lemma~\ref{admissiblePencil},
the Gr\"obner and Hessian certificates for the twelve nodes of the sextic
double solid, the smoothness of the quartic branch surface used in
Lemma~\ref{quarticSourceSmooth}, the homogenized quartic identity
\eqref{quarticStrictTransform}, and the singularity calculations in
Proposition~\ref{quarticBasket}.

\bibliographystyle{amsalpha}
\bibliography{Biblio}

\end{document}